\documentclass[11pt]{amsart}
\usepackage[margin=25mm]{geometry}

\usepackage{mystyle}
\usepackage[foot]{amsaddr}

\title{Strongly quasipositive links are concordant to infinitely many strongly quasipositive links}
\author{Paula Truöl}
\address{Max Planck Institute for Mathematics, Vivatsgasse 7, 53111 Bonn, Germany.}
\email{paulagtruoel@gmail.com}

\begin{document}

\def\subjclassname{\textup{2020} Mathematics Subject Classification}
\expandafter\let\csname subjclassname@1991\endcsname=\subjclassname
\subjclass{
57K10, 
20F36. 
}
\keywords{Concordance, strongly quasipositive links, satellite operation} 

\begin{abstract}
We show that every non-trivial strongly quasipositive link is smoothly concordant to infinitely many pairwise non-isotopic strongly quasipositive links. In contrast to our result, Baker conjectured that smoothly concordant strongly quasipositive fibered knots are isotopic. Our construction uses a satellite operation whose companion is a slice knot with maximal Thurston--Bennequin number -1.
\end{abstract}

\maketitle

\section{Introduction}

We study \emph{links} in the $3$-sphere $S^3$, i.e.~non-empty, oriented, closed, smooth $1$-dimensional submanifolds of $S^3$, up to (ambient) isotopy. By a fundamental theorem of Alexander \cite{alexander}, every link in $S^3$ can be represented as the closure of an $n$-braid for some $n \geq 1$. An $n$-\emph{braid} is an element of the \emph{braid group $B_n$ on $n$ strands}, whose classical presentation with $n-1$ generators $\sigma_1, \dots, \sigma_{n-1}$ and relations
\begin{align*}
\sigma_i \sigma_j = \sigma_j \sigma_i \quad \text{if } | i-j| \geq 2 \quad \text{and} \quad  \sigma_i \sigma_{i+1} \sigma_i = \sigma_{i+1} \sigma_{i} \sigma_{i+1} 
\end{align*}
was introduced by Artin \cite{artin_1925}. For a detailed discussion of braids and their closures, we refer the reader to \cite{birmanbrendle}. An $n$-braid is called \emph{quasipositive} if it is a finite product of conjugates of positive Artin generators $\sigma_i$ of $B_n$, and it is called \emph{strongly quasipositive} if the product consists only of \emph{positive band words} (or \emph{BKL generators} honoring the work of Birman--Ko--Lee \cite{birmankolee}) 
\begin{align*}
\sigij = \left(\sigma_i \cdots \sigma_{j-2} \right) \sigma_{j-1} \left(\sigma_i \cdots \sigma_{j-2} \right)^{-1} \qquad \text{for } 1 \leq i < j \leq n.
\end{align*}
Note that $\sigma_{i,i+1} = \sigma_i$.
A link is called \emph{(strongly) quasipositive} if it arises as the closure of a (strongly) quasipositive $n$-braid for some $n \geq 1$, respectively.

Two (ordered) links $L_0 = L_0^1 \cup \dots \cup L_0^r$ and $L_1 = L_1^1 \cup \dots \cup L_1^r$ of $r$ components are called \emph{(smoothly) concordant} if there exists a smoothly and properly embedded, oriented submanifold $A = A_1 \cup \dots \cup A_r$ of $S^3 \times [0,1]$ such that $A$ is diffeomorphic to a disjoint union of $r$ annuli $S^1 \times [0,1] $, $\partial A_i = L_0^i \times \{0\} \,\cup L_1^i \,\times \{1\}$, $i \in \{1, \dots, r\},$ and the induced orientation on $\partial A$ agrees with the orientation of $L_0$, but is the opposite one on $L_1$. 

Our main result is the following.

\begin{theorem}\label{thm:main}
Every non-trivial strongly quasipositive link is smoothly concordant to infinitely many pairwise non-isotopic strongly quasipositive links.
\end{theorem}

To discuss the context of this result, we will focus on knots in the rest of the introduction. \emph{Knots} are links with one connected component.

Quasipositive knots occur in complex geometry as transverse intersections of smooth algebraic curves in $\C^2$ with the $3$-sphere $S^3\subset \C^2$, which provides a geometric characterization of these knots~\cite{Rudolph_1983_alg_fcts,boileau_orevkov_2001}. 
Their study has recently attracted much attention because of its relation to contact geometry and Heegaard--Floer homology. For example, among fibered knots, strongly quasipositive knots are precisely those knots for which the contact structure on $S^3$ induced by the associated open book decomposition is the unique tight contact structure on $S^3$ \cite{hedden05}.
In the context of smooth concordance, (strongly) quasipositive knots are special. For example, it follows from Rudolph's slice--Bennequin inequality \cite{rudolphslice} that not every knot is concordant to a quasipositive knot, which is contrary to 
the behavior in the topological category \cite{borodzikFeller}. 

Consider the following inclusions:
\begin{align*}
\begin{split}
\{\text{algebraic knots}\} 
&\subset \{\text{positive knots}\}
\subset \{\text{strongly quasipositive knots}\}.
\end{split}
\end{align*}
\emph{Positive} knots are the knots that admit a diagram with only positive crossings. \emph{Algebraic} knots arise as links of isolated singularities of smooth algebraic plane curves in $\C^2$~\cite{milnor_book}. They are closures of positive braids (see \eg \cite[Theorem 12 in Section 8.3]{brieskornknoerrer}) and thus positive; the second inclusion is due to Rudolph and Nakamura~\cite{Rudolph_positiveLinksSQP,nakamura}. In contrast to \Cref{thm:main} and generalizing earlier results by Stoimenow \cite{Stoimenowpossign,StoimenowConc}, Baader--Dehornoy--Liechti \cite{Baader_2017} showed that every knot is (topologically and thus also smoothly) concordant to at most finitely many positive knots. Furthermore, Litherland~\cite{litherland} proved that algebraic knots are isotopic if they are concordant.

Indeed, according to a result of Baker \cite{Baker_2015}, either smoothly concordant, strongly quasipositive, fibered knots are isotopic, or the slice--ribbon conjecture is false. The slice--ribbon conjecture goes back to a question posed by Fox \cite{fox_1962} and asserts that every \emph{slice} knot is \emph{ribbon}, i.e~every knot that is concordant to the unknot bounds an immersed disk in $S^3$ with only ribbon singularities. A~pair of smoothly concordant, but non-isotopic, strongly quasipositive, fibered knots would provide a counterexample to this long-standing conjecture. The links constructed in the proof of \Cref{thm:main} are not fibered; see \Cref{rem:fibered}. Nevertheless, as far as we know, the following question---which is a weaker version of Baker's conjecture---is open.

\begin{question}\label{question}
Are there only finitely many strongly quasipositive, fibered knots in each smooth concordance class?
\end{question}

\begin{rem}\label{rem:BakerHedden}
In \cite{Baker_2015}, Baker explains a strategy personally communicated to him by Hedden which directly shows that, contrary to the conjectured result for strongly quasipositive, fibered knots, there are infinitely many pairs of (ribbon) concordant strongly quasipositive knots that are not isotopic. Indeed, the positive $k$-twisted Whitehead doubles of two (ribbon) concordant, non-isotopic knots provide examples of such pairs for negative,
sufficiently small $k$. This project began with the observation that, using the above idea but being careful about the choice of~$k$, one can construct an infinite family of concordant, pairwise non-isotopic strongly quasipositive knots by taking the positive $(-1)$-twisted Whitehead doubles of infinitely many concordant, pairwise non-isotopic knots, all of which have maximal Thurston--Bennequin number equal to $-1$. We explain this in more detail in \Cref{rem:oldconstrlonger}. Note that the statement of \Cref{thm:main} is stronger, since it shows that \emph{every} non-trivial strongly quasipositive knot is concordant to infinitely many strongly quasipositive knots.
\end{rem}

\begin{rem}
The non-triviality assumption in \Cref{thm:main} is necessary because there exists only one strongly quasipositive slice knot: the unknot. This follows from the fact that for strongly quasipositive knots the genus and the smooth $4$-genus coincide \cite{bennequin,rudolphslice}.
\end{rem}

\textbf{Organization.} In order to prove \Cref{thm:main}, we will first establish some notations and definitions regarding quasipositive Seifert surfaces and study examples of such surfaces in \Cref{sec:qpannuli}. In~\Cref{sec:tyingknots}, from two quasipositive Seifert surfaces $F_1$ and $F_2$ for links $\partial F_1$ and $\partial F_2$, we will construct a third surface whose boundary is a link that is concordant, but not isotopic to $\partial F_2$. The surface $F_1$ will be one of the quasipositive annuli from \Cref{sec:qpannuli}. We will finally prove \Cref{thm:main} in \Cref{sec:proofthm}.\\

\textbf{Acknowledgements.}
I would like to thank Peter Feller for his constant support and many fruitful discussions, the anonymous referees for their valuable comments, which also led to \Cref{rem:simplvolume}, and Michel Boileau for pointing out \Cref{question}. I would also like to thank Léo Benard, Pietro Capovilla, Lukas Lewark and Eric Stenhede for helpful discussions, Cara Hobohm for carefully reading an earlier version, and Arunima Ray for her question during the \emph{MATRIX-MFO Tandem Workshop 2136a: Invariants and Structures in Low-Dimensional Topology} at MFO Oberwolfach, which initiated this project. 
I would like to express my gratitude to the organizers of this workshop, as well as the organizers of \emph{Braids in Low-Dimensional Topology} at ICERM and \emph{Surfaces in 4-manifolds} at Le Croisic, for creating such a productive and inspiring environment. 


\section{Quasipositive Seifert surfaces and annuli}\label{sec:qpannuli}

We first define quasipositive Seifert surfaces. A \emph{Seifert surface} (for a link~$L$) is an oriented, compact surface in $S^3$ (with oriented boundary $L$) without closed components. Let $L$ be a link that arises as the closure $\widehat{\beta}$ of a strongly quasipositive braid $\beta \in B_n$, $n \geq 1$, which is a product of~$\ell \geq 0$ positive band words $\sigij$. We refer to such a product as a \emph{strongly quasipositive braid word}, and, despite a minor abuse of notation, also denote it by~$\beta$. 
There is a canonical Seifert surface of Euler characteristic $n-\ell$ for $L$ associated to the braid word $\beta$, which consists of $n$ copies of disjoint parallel disks and $\ell$~half-twisted bands connecting these disks~\cite{rudolph_braidedsurfaces,Rudolph_surfaces}.

See~\Cref{fig:annulus_m9_46_qp} for an example, ignoring the caption for now. The figure shows the canonical Seifert surface associated to the strongly quasipositive braid word 
$\alpha = \sigma_{1,6} \sigma_{3,8} \sigma_{2,5} \sigma_{1,4} \sigma_{3,7} \sigma_{2,6} \sigma_{5,8} \sigma_{4,7} \in B_8$ which represents a braid whose closure is a link of two components. The surface consists of eight disks and eight bands, so it has Euler characteristic zero. 

\begin{figure}[htbp]
  \centering
  \includegraphics[width=0.49\textwidth]{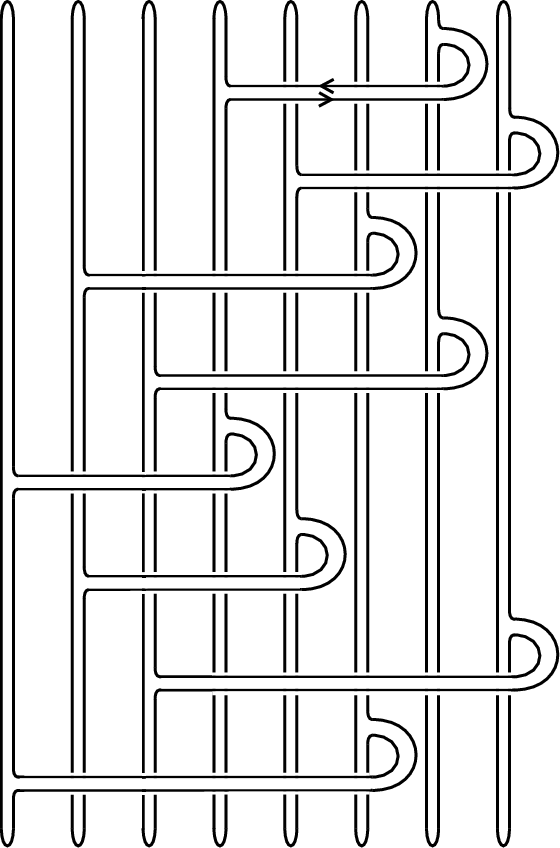}
\caption{The annulus $A\left(C, -1\right)$ for~$C = \mnine$ is ambient isotopic to $F(\alpha)$ for~$\alpha = \sigma_{1,6} \sigma_{3,8} \sigma_{2,5} \sigma_{1,4} \sigma_{3,7} \sigma_{2,6} \sigma_{5,8} \sigma_{4,7} \in B_8$.}  
\label{fig:annulus_m9_46_qp}
\end{figure}

We will denote the canonical Seifert surface associated to a braid word~$\beta$ by~$F(\beta)$. We call any Seifert surface~$F$ for a link~$L= \partial F$ \emph{quasipositive} if, for some strongly quasipositive braid word~$\beta \in B_n$, $n \geq 1$, $F$ is ambient isotopic to $F(\beta)$. We will be particularly interested in certain quasipositive annuli, such as the one shown in \Cref{fig:annulus_m9_46_qp}. For our construction in \Cref{sec:tyingknots}, in particular, we will need a quasipositive annulus for which the core curve 
is a slice knot.


Now, let $C$ be a nontrivial (\ie not the unknot) slice knot with maximal Thurston--Bennequin number~$\TB(C)=-1$, \eg the mirror of the knot $9_{46}$ from Rolfsen's knot table \cite{ng_Thurston-Bennequin,knotinfo}, which we denote by $\mnine$. For our purposes, we could use any such knot $C$, but for the sake of concreteness of our illustrations we will fix $C = \mnine$ in the entire text.

Recall that every knot $K$ has a Legendrian representative (which is at each point in $S^3$ tangent to the $2$-planes of the standard contact structure on $S^3$) and its maximal Thurston--Bennequin number~$\TB\left(K\right)$ is defined as 
\begin{align*}
\TB\left(K\right) = \max \{ \tb(\mathcal{L}) \mid \mathcal{L} \text{ is a Legendrian representative of } K\}.
\end{align*}
Here, for a Legendrian knot $\mathcal{L}$, $\tb(\mathcal{L})$ denotes its Thurston--Bennequin number; see \eg \cite{etnyre} for a definition.

\Cref{fig:LegRep} shows the front projection of a Legendrian representative $\mathcal{L}$ of~$\mnine$ with $\tb(\mathcal{L}) = -1$. There is a Lagrangian concordance between the Legendrian unknot $\mathcal{U}$ with $\tb(\mathcal{U}) = -1$ and $\mathcal{L}$; see~\cite[Figure 4]{Chantraine15}. In particular, the knot $\mnine$ is slice, and since $\TB(K) \leq -1$ for every slice knot $K$ \cite{rudolph_obstrtoslice}, this implies $\TB\left(\mnine\right)=-1$.

\begin{figure}[htbp]
  \centering
  \includegraphics[width=1\textwidth]{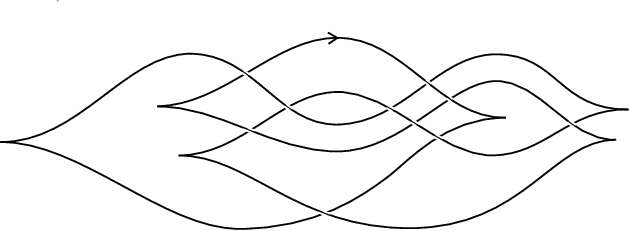}
\caption{Front projection of a Legendrian representative $\mathcal{L}$ of $\mnine$ with Thurston--Bennequin number $\tb(\mathcal{L}) = -1$; cf. \cite[Figure~1]{Chantraine15}.}  
\label{fig:LegRep}
\end{figure}

For a knot $K$ and an integer $k$, following Rudolph's notation \cite{rudolph_quasipositive_annuli}, let $A\left(K, k\right)$ denote an annulus of type $K$ with $k$ full twists, \ie $A\left(K, k\right) \subset S^3$ is an annulus with $K \subset \partial A\left(K,k\right)$ such that~$\lk\left(K, \partial A\left(K,k\right) \setminus K\right) = -k$, where $\lk$ denotes the linking number. We have
\begin{align*}
\sup \left\{ k \mid A\left(K,k\right) \text{ is quasipositive} \right\} = \TB\left(K\right) \qquad \text{\cite[Proposition 1]{rudolph_obstrtoslice}}.
\end{align*}
Hence, for every knot $K$ with $\TB\left(K\right)=-1$, the annulus $A\left(K, -1\right)$ is quasipositive, in particular for $K = C = \mnine$.
This is the key observation of this subsection. It implies the existence of a strongly quasipositive braid word $\alpha \in B_m$ for some $m \geq 1$ such that $A\left(C, -1\right)$, the annulus of type~$C$, or equivalently, the annulus with core curve $C$, is ambient isotopic to~$F(\alpha)$. 
For example, we can choose 
\begin{align}\label{eq:alphabraidword}
\alpha = \sigma_{1,6} \sigma_{3,8} \sigma_{2,5} \sigma_{1,4} \sigma_{3,7} \sigma_{2,6} \sigma_{5,8} \sigma_{4,7}  \in B_8;
\end{align}
see \Cref{fig:annulus_m9_46_qp}. 
Note that for any knot $K$ with $\TB\left(K\right)=k$, a quasipositive braid diagram for $A\left(K, k\right)$ can be found by taking the Legendrian ribbon of a Legendrian representative $\mathcal{L}$ of $K$ with $\tb(\mathcal{L}) = k$ \cite{Rudolph_II,rudolph_quasipositive_annuli}. 

\section{Tying knots into positive bands of quasipositive Seifert surfaces preserving quasipositivity}\label{sec:tyingknots}

Let $L$ be a link other than an unlink and let $F$ be a quasipositive Seifert surface for $L$. As in \Cref{sec:qpannuli}, let $C= \mnine$ such that the annulus $A\left(C, -1\right)$ is quasipositive. As mentioned above, for $C$ we could also use every other nontrivial slice knot with $\TB(C)=-1$.

In this section, starting from the quasipositive Seifert surfaces $A\left(C, -1\right)$ and $F$, we will define a new quasipositive Seifert surface $F^\prime$, which has as boundary a link that is concordant, but not isotopic to $L = \partial F$. To do this, for both $A\left(C, -1\right)$ and $F$ choose strongly quasipositive braid words~$\alpha \in B_m$ and $\beta \in B_n$ for $m, n \geq 1$, respectively, such that $A\left(C, -1\right)$ is ambient isotopic to $F(\alpha)$ and $F$ is ambient isotopic to $F(\beta)$. For example, we can and will choose $\alpha$ as in \eqref{eq:alphabraidword} from \Cref{sec:qpannuli}. Let $\beta = \prod_{k=1}^\ell \sigma_{i_k, j_k} \in B_n$ for some $1 \leq i_k < j_k \leq n$, $n \geq 1$. We can put the surfaces $F(\alpha)$ and~$F(\beta)$ in split position in $S^3$ (as sketched in \Cref{fig:a}). Concretely, we can take $F(\alpha)$ to lie in the lower hemisphere and $F(\beta)$ to lie in the upper hemisphere of $S^3$, respectively. Then we can choose a cylinder $Z \subset S^3$ such that the bands $B_\alpha$ of $F(\alpha)$ and $B_\beta$ of $F(\beta)$ corresponding to the positive band words $\sigma_{4,7}$ and $\sigma_{i_1, j_1}$, respectively, intersect $Z$ as indicated in the upper part of \Cref{fig:cylinderwithbands} (ignoring $\varphi(\gamma)$ for now). More precisely, we can choose a cylinder $Z \subset S^3$ and an orientation-preserving diffeomorphism $\varphi \colon Z \to D^2 \times [0,1]$ such that $Z \cap F(\alpha) = Z\cap B_\alpha$, $Z \cap F(\beta) =Z\cap  B_\beta$ and $\varphi$ maps
\begin{align}\label{eq:cylindertriple}
\begin{split}
& Z\cap B_\alpha \xrightarrow[\varphi]{\cong} 
\left[-\frac{2}{3}, -\frac{1}{3}\right] \times [0,1], \qquad Z \cap \partial B_\alpha \xrightarrow[\varphi]{\cong} \left\{-\frac{2}{3}, -\frac{1}{3}
 \right\} \times [0,1],\\
&Z\cap  B_\beta \xrightarrow[\varphi]{\cong} \left[\frac{1}{3}, \frac{2}{3}\right] \times [0,1], \qquad  Z \cap \partial B_\beta \xrightarrow[\varphi]{\cong}
\left\{\frac{1}{3}, \frac{2}{3}
 \right\} \times [0,1],
 \end{split}
\end{align}
where $X \stackrel{\cong}{\rightarrow} Y$ indicates an orientation-preserving diffeomorphism. Here, $\left[a,b\right] \subset D^2$ denotes the straight line segment connecting $a$ and $b$ in the closed unit disk $D^2\subset\C$.
We choose the orientations on $D^2 \times [0,1]$ and $\left[a,b\right] \times [0,1]$ induced by the standard orientations on $\C \times \R \cong \R^3 \subset S^3$ and $\R^2$, respectively; the orientation on $\left[a,b\right] \times [0,1]$ also induces an orientation on $\left\{a,b\right\} \times [0,1]$.
We claim that we can choose $Z$ and $\varphi$ as above such that there exists a simple closed curve $\gamma$ in $S^3 \setminus F(\beta)$ that wraps once around the band $B_\beta$ of $F(\beta)$ corresponding to $\sigma_{i_1, j_1}$ and that does not bound a disk in $S^3 \setminus \partial F(\beta)$. More precisely, we claim the following.

\begin{figure}[htbp]
  \centering
  \begin{psfrags}
  \psfrag{a}{\textcolor{myred}{$\varphi(\gamma)$}}
  \includegraphics[width=0.35\textwidth]{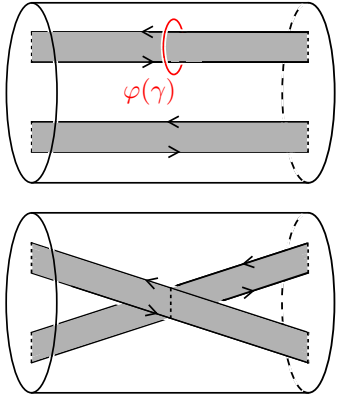}
\end{psfrags}
\caption{Top: 
The triple $\big(D^2 \times [0,1], \left[-\frac{2}{3}, -\frac{1}{3}\right] \times [0,1] \cup \left[\frac{1}{3}, \frac{2}{3}\right] \times [0,1]$, $\left\{\pm\frac{1}{3}, \pm\frac{2}{3}\right\} \times [0,1]\big)$, which is mapped to $\left(Z, Z\cap \left(B_\alpha \cup B_\beta\right), Z \cap \partial \left(B_\alpha \cup B_\beta\right)\right)$ via $\varphi$. The red curve depicts $\varphi(\gamma) \subset D^2 \times [0,1]$; see (ii) in \Cref{lem:cylinder}. Bottom: $B^\prime \subset D^2 \times [0,1] $ for $B^\prime$ as defined in \eqref{eq:defnD}.
}  
\label{fig:cylinderwithbands}
\end{figure}

\begin{lem}\label{lem:cylinder}
Let $F(\beta)$ be the quasipositive Seifert surface associated to a strongly quasipositive braid word $\beta$ such that $\partial F(\beta)$ is not an unlink. We can choose a cylinder $Z \subset S^3$ and an orientation-preserving diffeomorphism $\varphi \colon Z \to D^2 \times [0,1]$ such that 
\begin{enumerate}[label=(\roman*)]
\item $Z \cap F(\alpha) = Z\cap B_\alpha$ and $Z \cap F(\beta) =Z\cap  B_\beta$ verify \eqref{eq:cylindertriple}, where $B_\alpha$ and $B_\beta$ are the bands corresponding to the last positive band word of $F(\alpha)=A(C,-1)$ and the first positive band word of $F(\beta)$, respectively,
\item \label{lemii}
$\gamma \eqdef \varphi^{-1}(  C_{\frac{1}{3}}(\frac{1}{2})\times \{\frac{1}{2}\})$ is an unknot in $S^3 \setminus F(\beta)$ that wraps once around the band $B_\beta$ and that does not bound a disk in $S^3 \setminus \partial F(\beta)$. Here $C_{\frac{1}{3}}(\frac{1}{2}) \subseteq D^2$ denotes the circle with center~$\frac{1}{2}$ and radius $\frac{1}{3}$. The situation is shown in the upper part of \Cref{fig:cylinderwithbands} with $\varphi(\gamma)$ in red. 
\item \label{lemiii}
Moreover, either the two components of $\partial B_\beta \cap \partial F(\beta)$ belong to two different components of the link $\partial F(\beta)$ or we can choose $Z$ and $\varphi$ such that there exists a quasipositive Seifert surface~$G$ with the following properties: $G$ is a connected component of $F(\beta)$ with boundary~$J=\partial G$ one of the components of $\partial F(\beta)$, $B_\beta \subset G$ (in particular, $\partial B_\beta \cap \partial F(\beta) \subset J$) and $\gamma$ does not bound a disk in $S^3 \setminus J$.
\end{enumerate}
\end{lem}

We postpone the proof of \Cref{lem:cylinder} until the end of this section. Now, let $B = Z\cap \left(B_\alpha \cup B_\beta\right)$ for a cylinder $Z$ and bands $B_\alpha$ and $B_\beta$ as in \Cref{lem:cylinder} and define 
$F^\prime =\left(\left( F(\alpha ) \cup F(\beta) \right)\setminus B \right)\cup \varphi^{-1}\left(B^\prime\right), $ where $B^\prime$ is as in the lower part of \Cref{fig:cylinderwithbands}.
More precisely, let
\begin{align}\label{eq:defnD}
\begin{split}
B^\prime = &\left\{\left(a+t, t\right)\mid t \in[0,1],\, a \in \left[-\frac{2}{3}, -\frac{1}{3}\right]  \right\} \\
\cup &\left\{\left(a-t+ti, t\right)\mid t \in\left[0,\frac{1}{2}\right], \,a \in \left[\frac{1}{3}, \frac{2}{3}\right]  \right\} \\
\cup &\left\{\left(a-t+(1-t)i, t\right)\mid t \in\left[\frac{1}{2},1\right],\, a \in \left[\frac{1}{3}, \frac{2}{3}\right]  \right\}\subseteq D^2 \times [0,1].
\end{split}
\end{align}
In the definition of $B^\prime$ in \eqref{eq:defnD} (and only there), $i\in  \C$ denotes the imaginary unit. We smooth the corners of $\varphi^{-1}\left(B^\prime\right)$ to obtain a smooth surface $F^\prime$ and claim the following.

\begin{lem}\label{lem:operation}
Let $Z$, $\varphi$, $\gamma$, $B_\alpha$ and $B_\beta$ be given by \Cref{lem:cylinder}. Then the surface $$F^\prime =\left(\left( F(\alpha ) \cup F(\beta) \right)\setminus B \right)\cup \varphi^{-1}\left(B^\prime\right)$$ (with smoothed corners), where $B = Z\cap \left(B_\alpha \cup B_\beta\right)$ and $B^\prime$ is defined as in \eqref{eq:defnD}, is a quasipositive Seifert surface for a link $L^\prime =\partial F^\prime$ that is concordant, but not isotopic to $L=\partial F(\beta)$.
\end{lem}

\begin{proof}[Proof of \Cref{lem:operation}]
We will show the following three claims separately.
\begin{claim}\label{claim:quasipositivity}
The surface $F^\prime $ is quasipositive.
\end{claim}
\begin{claim}\label{claim:concordance}
The link $L^\prime=\partial F^\prime$ is concordant to the link $L=\partial F(\beta)$. 
\end{claim} 
\begin{claim}\label{claim:simplvolume}
The links $L$ and $L^\prime$ are not isotopic. 
\end{claim}

\begin{claimproof}{Proof of \Cref{claim:quasipositivity}}
Quasipositivity of $F^\prime$ can be shown using an isotopy as depicted in \Cref{fig:b,fig:c,fig:d}. A strongly quasipositive braid word $\delta$ for a quasipositive Seifert surface $F(\delta)$ that is ambient isotopic to $F^\prime$ can then be read off in \Cref{fig:d}.
\end{claimproof}

\begin{claimproof}{Proof of \Cref{claim:concordance}}
Observe that the surface $F^\prime$ is obtained from $F(\beta)$ by tying the knot $C$ with framing $0$ into the band $B_\beta$ of $F(\beta)$; see \Cref{fig:operationistyingknotinband}. This amounts to realizing the boundary of~$F^\prime$ as a satellite link with pattern $\partial F(\beta)$ and companion $C$. We explain this in detail. The link~$\partial F(\beta)$ can be viewed as a link in the solid torus $S^3 \setminus \nu(\gamma)$ given by the complement of an open tubular neighborhood~$\nu(\gamma)$ of $\gamma$ in $S^3$. We identify this solid torus with the standard solid torus $V = D^2 \times S^1 \subset S^3\subset \C^2$ by an orientation-preserving diffeomorphism that takes the preferred longitude of $S^3 \setminus \nu(\gamma)$ to $\{1\} \times S^1 \subset V$. Then $L^\prime = \partial F^\prime$ arising as a satellite link with pattern $\partial F(\beta)$ and companion $C$ means that $L^\prime$ is the image of $\partial F(\beta) \subset S^3 \setminus \nu(\gamma)\cong  V$ under an orientation-preserving embedding $ h\colon V =D^2 \times S^1  \hookrightarrow S^3$ that maps $\{0\}\times S^1$ to $C$ and $\{1\}\times S^1 $ to a curve that has linking number $0$ with $h\left(\{0\} \times S^1\right)$. For more details on satellite constructions, and in particular on the terms used here, see \cite[Sections 2E and 4D]{rolfsen_2003}. Our choices ensure that $h$ is faithful in Rolfsen's terminology and that the companion is really $C$ and not the reverse of $C$.

\captionsetup[subfigure]{labelfont=normalfont, labelformat = simple}
\begin{figure}[htbp]
     \centering
     \begin{subfigure}[b]{0.49\textwidth}
         \centering
         \begin{psfrags}
  \psfrag{a}{\textcolor{myred}{$\gamma$}}
\psfrag{b}{\textcolor{myblue}{$F(\beta)$}}
\psfrag{c}{$F(\alpha)$}
         \includegraphics[width=\textwidth]{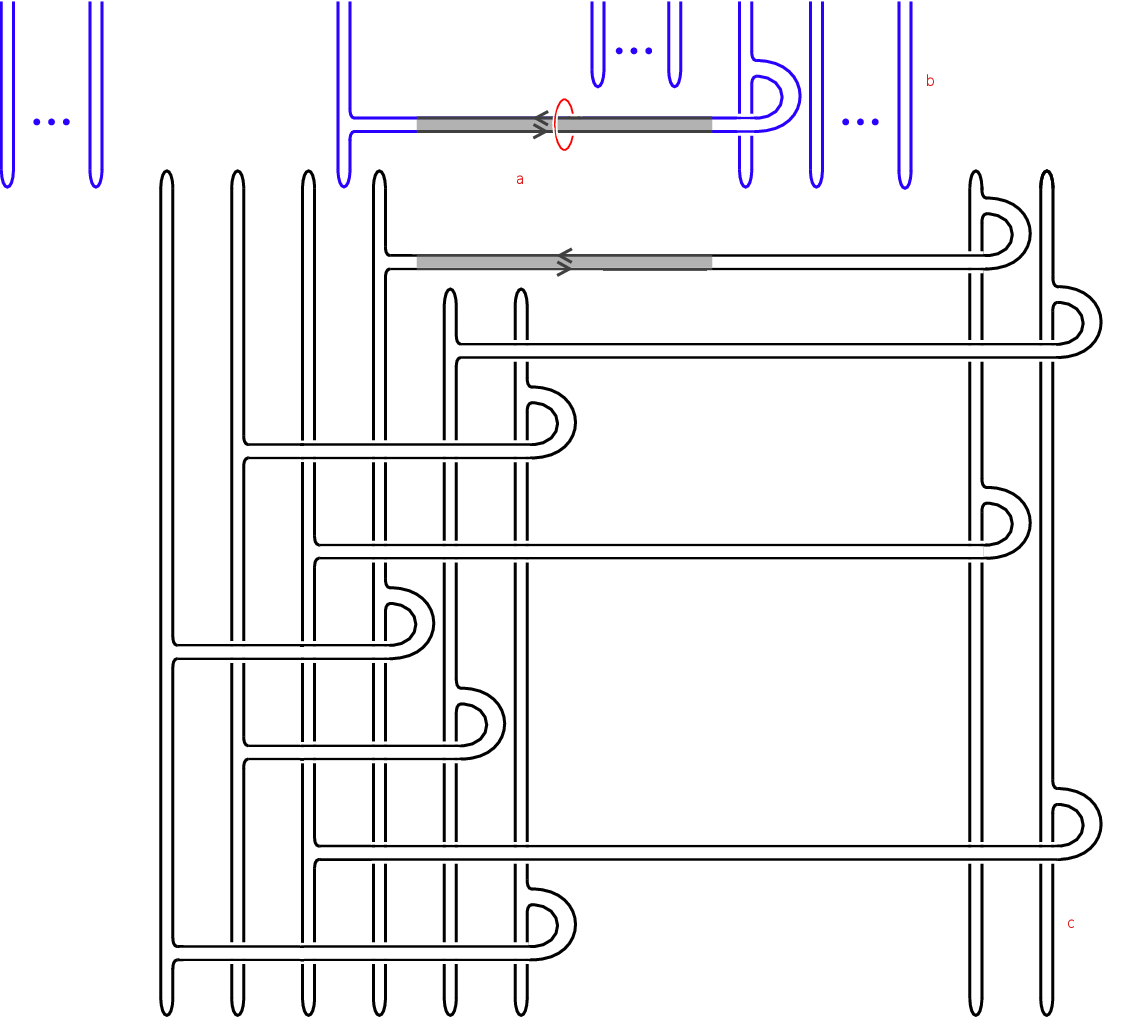}
         \end{psfrags}
         \caption{}
         \label{fig:a}
     \end{subfigure}
     \hfill
     \begin{subfigure}[b]{0.49\textwidth}
         \centering
         \begin{psfrags}
  \psfrag{c}{$F^\prime$}
         \includegraphics[width=\textwidth]{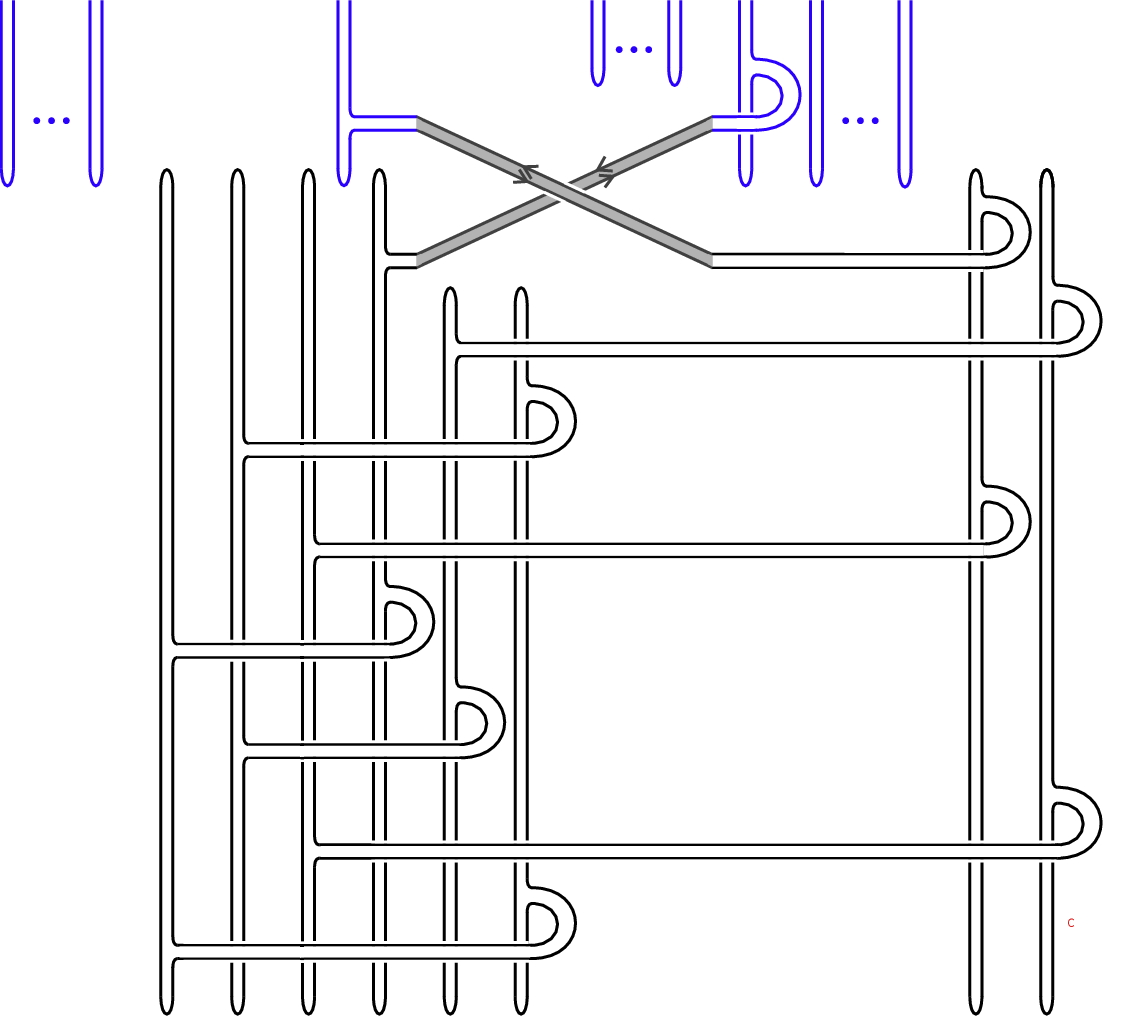}
         \end{psfrags}
         \caption{}
         \label{fig:b}
     \end{subfigure}
     \newline
     \begin{subfigure}[t]{0.49\textwidth}
         \centering
         \includegraphics[width=\textwidth]{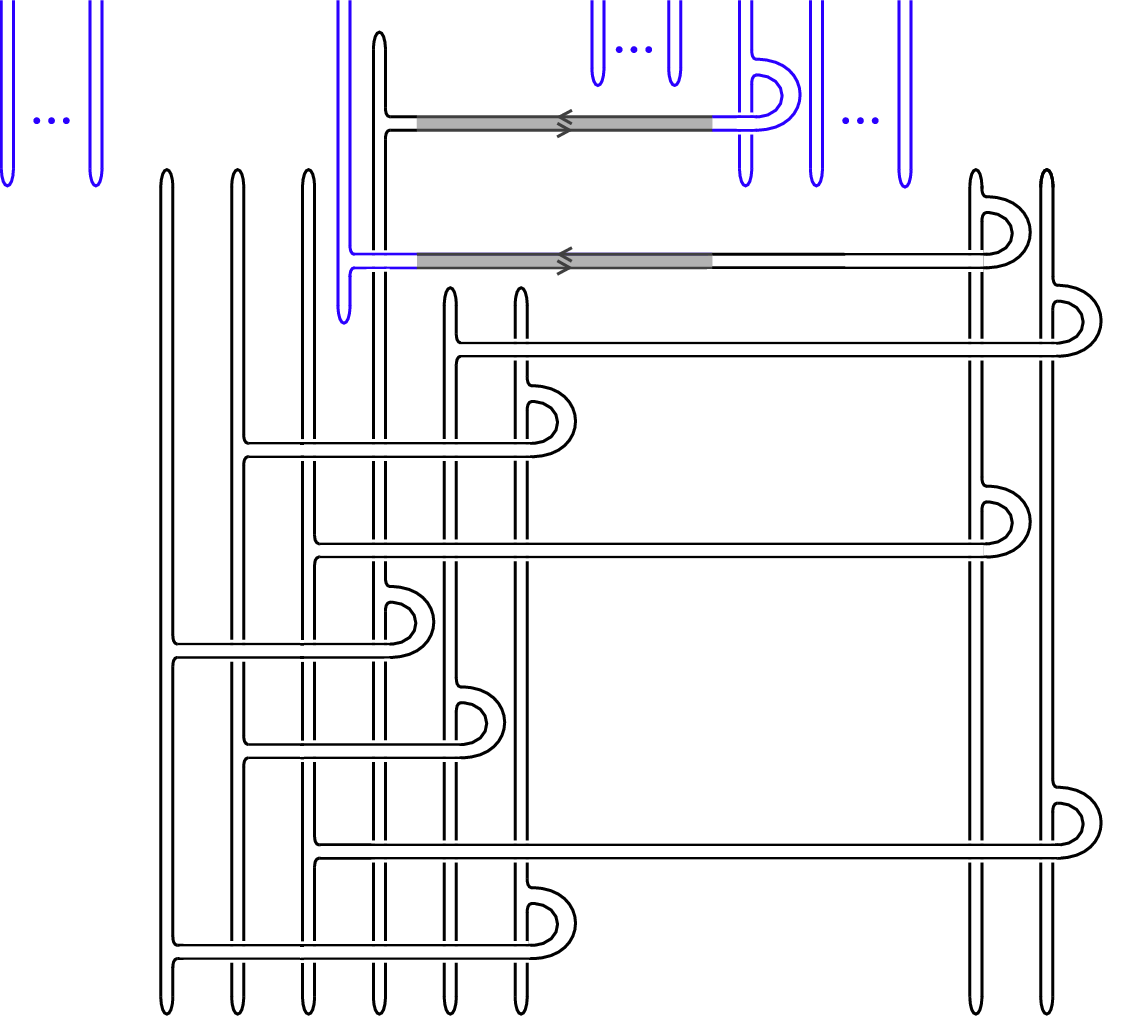}
         \caption{}
         \label{fig:c}
     \end{subfigure}
          \hfill
     \begin{subfigure}[t]{0.49\textwidth}
         \centering
         \includegraphics[width=\textwidth]{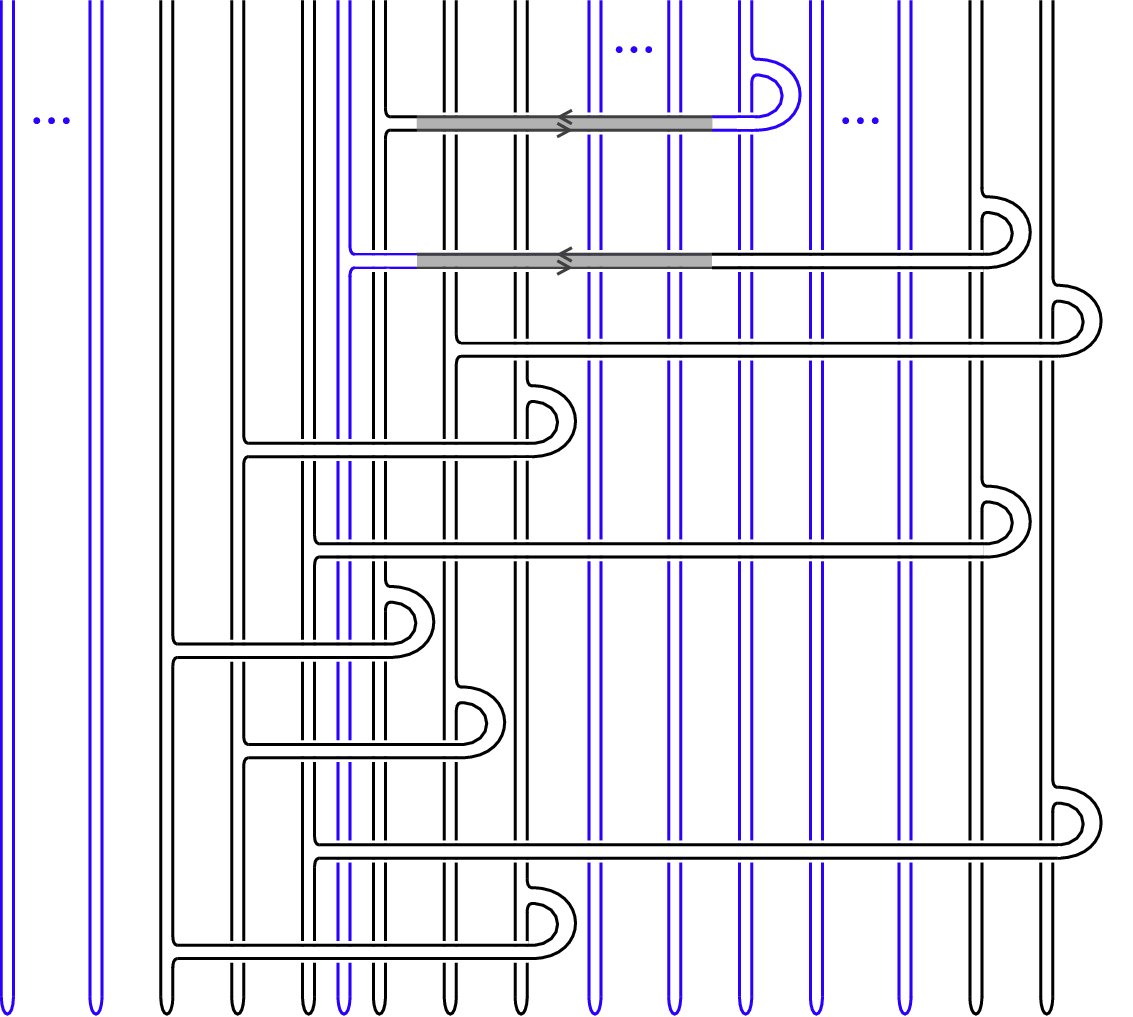}
         \caption{}
         \label{fig:d}
     \end{subfigure}
\caption{Quasipositivity of the surface $F^\prime$ defined from $F(\alpha)$ and $F(\beta)$ (see \Cref{lem:operation}). The surface $F(\alpha)$ is shown in black, the surface $F(\beta)$ in blue and $Z\cap (F(\alpha) \cup F(\beta))$ in grey. Subfigure \ref{fig:a} shows $F(\alpha)$, $F(\beta)$ and $\gamma$ as in \Cref{lem:operation}. Subfigure \ref{fig:b} shows the surface $F^\prime$ which is ambient isotopic to the canonical quasipositive Seifert surface in \ref{fig:d}; an intermediate stage of such an isotopy is shown in~\ref{fig:c}.}  
\label{fig:SchematicOperationPreservesQuasipositivity}
\end{figure}

\captionsetup[subfigure]{labelfont=normalfont, labelformat = simple}
\begin{figure}[htbp]
     \centering
     \begin{subfigure}[b]{0.49\textwidth}
         \centering
         \begin{psfrags}
\psfrag{b}{\textcolor{myblue}{$F(\beta)$}}
\psfrag{c}{$F(\alpha)$}
\psfrag{e}{C}
         \includegraphics[width=\textwidth]{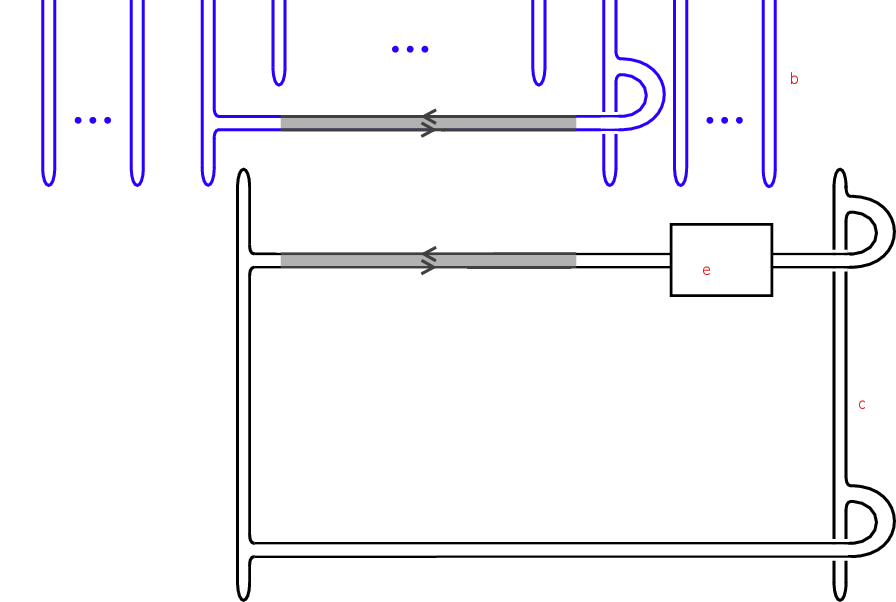}
         \end{psfrags}
         \caption{}
         \label{fig:5a}
     \end{subfigure}
     \hfill
     \begin{subfigure}[b]{0.49\textwidth}
         \centering
         \begin{psfrags}
  \psfrag{c}{$F^\prime$}
  \psfrag{e}{C}
         \includegraphics[width=\textwidth]{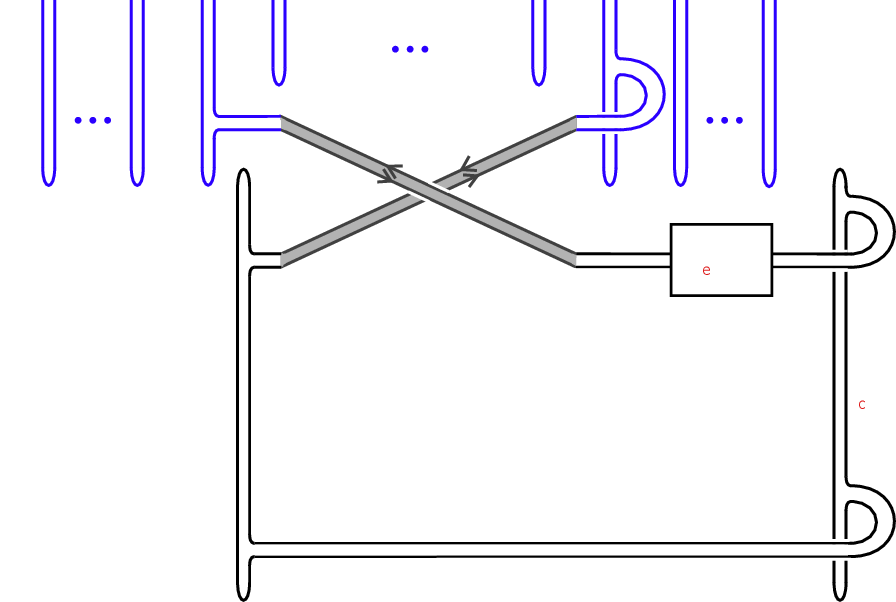}
         \end{psfrags}
         \caption{}
         \label{fig:5b}
     \end{subfigure}
     \newline
     \begin{subfigure}[b]{0.49\textwidth}
         \centering
         \begin{psfrags}
    \psfrag{e}{C}
         \includegraphics[width=\textwidth]{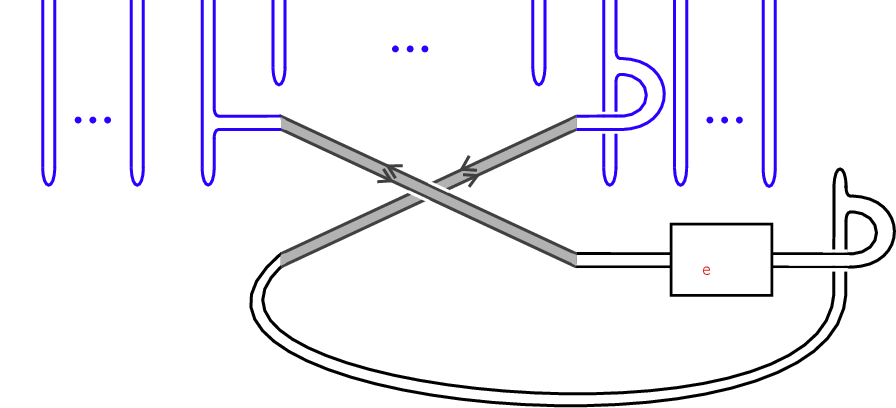}
		\end{psfrags}
         \caption{}
         \label{fig:5c}
     \end{subfigure}
          \hfill
     \begin{subfigure}[b]{0.49\textwidth}
         \centering
         \begin{psfrags}
    \psfrag{e}{C}
         \includegraphics[width=\textwidth]{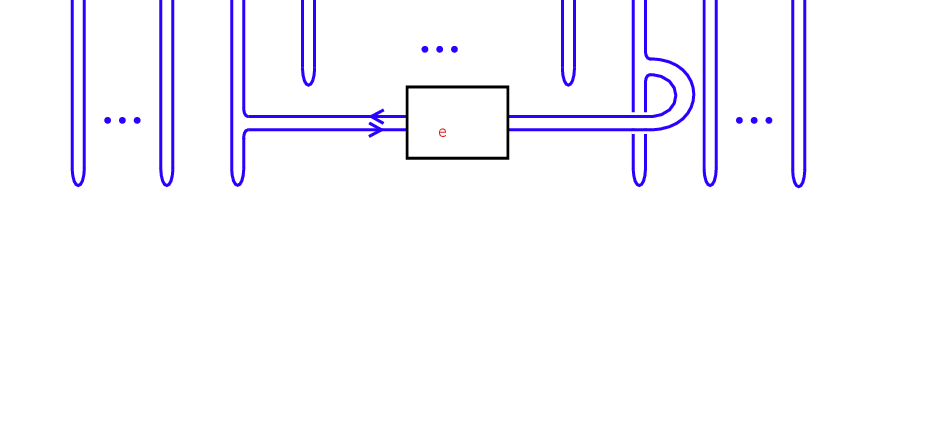}
		\end{psfrags}
         \caption{}
         \label{fig:5d}
     \end{subfigure}
\caption{The surface $F^\prime$ is obtained from $F(\beta)$ by tying the knot $C$ into the band~$B_\beta$ corresponding to the positive band word $\sigma_{i_1, j_1}$ of $\beta$. Subfigure \ref{fig:5a} shows a schematic representation of the surfaces $F(\alpha)$ and $F(\beta)$, and subfigure \ref{fig:5b} one of~$F^\prime$. Subfigures \ref{fig:5c} and \ref{fig:5d} indicate an ambient isotopy between $F^\prime$ and the surface~$F(\beta)$ with the knot $C$ tied into $B_\beta$ with framing $0$.}  
\label{fig:operationistyingknotinband}
\end{figure}

Now, it is a standard fact of concordance theory that, since $C$ is slice, \ie concordant to the unknot~$U$, there is a concordance between $L^\prime$, which is the satellite link with pattern $L$ and companion~$C$, and $L$, which is the satellite link with the same pattern but companion $U$. Indeed, if~$C$ and~$U$ are concordant via an annulus $A \cong S^1 \times [0,1] \subset S^3 \times [0,1]$, then we can identify $\left(S^3 \setminus \nu (\gamma) \right)\times [0,1]$ with a tubular neighborhood of $A$ in $S^3 \times [0,1]$ and the image of $L \times [0,1]$ in~$S^3 \times [0,1]$ under this identification provides us with a concordance between the two satellite links.
\end{claimproof}

\begin{claimproof}{Proof of \Cref{claim:simplvolume}}
On the other hand, we claim that since $C$ is not isotopic to $U$, the satellite links~$L^\prime = \partial F^\prime$ and $L=\partial F(\beta)$ are not isotopic. To prove this, we distinguish two cases. Note that the two components of $\partial B_\beta \cap \partial F(\beta)$ do not necessarily belong to the same component of $\partial F(\beta)$.

We first assume that they do, which is the case, for example, if $\partial F(\beta)$ is a knot; and we can further assume that $\gamma$ does not bound a disk in the complement of this component~$J$ of $\partial F(\beta)$ in $S^3$ (see \ref{lemiii} in \Cref{lem:cylinder}). In this case our satellite operation modifies up to ambient isotopy only the component~$J$ of $\partial F(\beta)$ by applying a satellite operation with companion $C$ and pattern~$J$. It follows from work of Kouno--Motegi \cite[Theorem 1.1]{kounomotegi} that this satellite operation on the non-isotopic knots $C$ and $U$ produces non-isotopic components $J(C)$ and $J(U)=J$ of $L=\partial F^\prime$ and~$L^\prime=\partial F(\beta)$. Here we need the assumptions on $\gamma$, which imply that the pattern $J$ we use in the satellite construction has wrapping number strictly greater than $1$. The \emph{wrapping number}~$\omega_V(P)$ of a pattern $P$ in the solid torus $V=D^2 \times S^1$ is the minimal geometric intersection number of $P$ and a generic meridional disk of $V$. Recall that we can consider $\partial F(\beta)$ and hence also its component~$J$ as a link/knot in the solid torus $S^3 \setminus \nu(\gamma)$, which we identify with $V$ by an orientation-preserving diffeomorphism that takes the preferred longitude of $S^3 \setminus \nu(\gamma)$ to $\{1\} \times S^1$. Then $\gamma$ not bounding a disk in $S^3\setminus J$ implies that~$J$ geometrically intersects nontrivially every meridional disk in $S^3\setminus \nu(\gamma) \cong V$, so $\omega_V(J) \neq 0$. Since the algebraic winding number of $J$ in $S^3\setminus \nu(\gamma) \cong V$ is zero (thus even), we get $\omega_V(J) > 1$.

Now, suppose that the two components of $\partial B_\beta \cap \partial F(\beta)$ belong to two different components $L_1$ and $L_2$ of the link $\partial F(\beta)$. The satellite operation then has the effect of tying $C$ into both of these components (up to orientation), \ie the resulting link has components $L_1 \# C$ and $L_2 \# C^r$, where~$C^r$ denotes $C$ with the reversed orientation, and all other components unchanged. Note that for our particular choice $C = \mnine$ we have $C = C^r$ \cite{knotinfo}. We clearly obtain a link $\partial F^\prime$ that is not isotopic to $\partial F(\beta)$.
\end{claimproof}
This concludes the proof of \Cref{lem:operation}.
\end{proof}

\begin{rem}\label{rem:fibered}
The link $\partial F^\prime$ constructed in \Cref{lem:operation} is not fibered. Indeed, we can use the Seifert--van Kampen theorem to show that on the level of fundamental groups $\pi_1(S^3 \setminus C)\hookrightarrow \pi_1(S^3 \setminus F^\prime)$, thus~$\pi_1(S^3 \setminus F^\prime)$ is not free. The constructed Seifert surface $F^\prime$ is thus not a fiber surface for $\partial F^\prime$, and since it is a Seifert surface with maximal Euler characteristic for $\partial F^\prime$, the link $\partial F^\prime$ cannot be fibered.
\end{rem}

\begin{rem}\label{rem:simplvolume}
As an anonymous referee pointed out, in \Cref{lem:operation} we could actually show that the constructed link $L\pr = \partial F^\prime$ has strictly larger complexity than the link $L=\partial F(\beta)$ in the following sense. For a knot $K$, let $ \lVert K \rVert$ denote the sum of the hyperbolic volumes of the hyperbolic pieces of the JSJ decomposition \cite{jacoshalen,johannson} of the knot exterior (which up to multiplication by a constant is equal to the simplicial volume of (the exterior of)~$K$  \cite{Soma1981,gromovvolume,thurston2022geometry}), and for a link $L = L^1 \cup \dots \cup L^r$ of $r$ components, let $c(L) = \sum_{j=1}^r \lVert L^j \rVert.$
Then, based on the arguments in the proof of Claim 3 above, we can show that $c( L\pr )  > c( L )$. We will briefly explain why below. Note that this stronger statement directly implies \Cref{thm:main} by an iterative application, but carrying out the details of the argument did not seem to strictly shorten our original argument, so we kept it. Indeed, in the case where the two components of $\partial B_\beta \cap \partial F(\beta)$ belong to the same component of the link $L=\partial F(\beta)$, we have $\lVert J(C) \rVert   > \lVert J \rVert$ and hence $c( L^\prime)  > c( L )$, since the companion $C=\mnine$ of the satellite operation is hyperbolic \cite{knotinfo}, and since the simplicial volume is additive if two $3$-manifolds are glued along a torus that is incompressible in both manifolds; see~\cite{Soma1981,gromovvolume,kuessner} and \cite[Theorem 3]{bucheretal} for a more general version on $n$-manifolds. In the other case, we can use the additivity of the simplicial volume under the connected sum of knots \cite[Theorem 2]{Soma1981} and again the hyperbolicity of~$C=\mnine$.
\end{rem}

We end this section with the deferred technical proof of \Cref{lem:cylinder}. The reader may choose to skip it on first reading. 

\begin{proof}[Proof of \Cref{lem:cylinder}]
We will prove a slightly abbreviated version. 
\begin{claim}\label{claim:mainlem}
Let $F(\beta)$ be the quasipositive Seifert surface associated to a strongly quasipositive braid word $\beta$ such that $\partial F(\beta)$ is not an unlink. We can choose a cylinder $Z \subset S^3$ and an orientation-preserving diffeomorphism $\varphi \colon Z \to D^2 \times [0,1]$ such that 
\begin{enumerate}[label=(\roman*)]
\item $Z \cap F(\beta) =Z\cap  B_\beta$ verifies \eqref{eq:cylindertriple}, where $B_\beta$ is a band of $F(\beta)$,
\item 
$\gamma \eqdef \varphi^{-1}(  C_{\frac{1}{3}}(\frac{1}{2})\times \{\frac{1}{2}\})$ 
is an unknot in $S^3 \setminus F(\beta)$ that wraps once around the band $B_\beta$ and that does not bound a disk in $S^3 \setminus \partial F(\beta)$.
\item 
Moreover, either the two components of $\partial B_\beta \cap \partial F(\beta)$ belong to two different components of the link $\partial F(\beta)$ or we can choose $Z$ and $\varphi$ such that there exists a quasipositive Seifert surface~$G$ with the following properties: $G$ is a connected component of $F(\beta)$ with boundary~$J=\partial G$ one of the components of $\partial F(\beta)$, $B_\beta \subset G$ and $\gamma$ does not bound a disk in~$S^3 \setminus J$.
\end{enumerate}
\end{claim}

Note that it is enough to show \Cref{claim:mainlem} for the following two reasons. First, up to conjugation of $\beta$ or ambient isotopy of $F(\beta)$, that is, up to a different choice of $Z$ and $\varphi$, we can choose any positive band word of $\beta$ to be the first one. Second, the position of $F_\alpha$ and $B_\alpha$ in \Cref{lem:cylinder} can be easily arranged by an ambient isotopy of $F(\alpha)$ as soon as $F(\beta)$ satisfies the conclusions of \Cref{claim:mainlem}. 

\begin{claimproof}{Proof of \Cref{claim:mainlem}}
Let $\beta = \prod_{k=1}^\ell \sigma_{i_k, j_k} \in B_n$ for some $n \geq 1$.
We claim that one of the following is true.
\begin{itemize}[leftmargin=1.7cm]
\item[Case 1:] There exists a half-twisted band in $F(\beta)$ corresponding to one of the positive band words~$\sigma_{i_k, j_k}$, $k \in \{1, \dots, \ell\}$, of $\beta$ such that the two arcs obtained as intersection of the boundary of this band with $\partial F(\beta)$ belong to two different components of the link $\partial F(\beta)$.
\item[Case 2:] The surface $F(\beta)$ is a disjoint union of quasipositive Seifert surfaces each of which has only one boundary component.
\end{itemize}
Here is the argument why: If the half-twisted bands in $F(\beta)$ are such that for each of them the two arcs obtained as intersection of the boundary of the band with $L=\partial F(\beta)$ belong to the same component of $L$, then for each of the disks in $F(\beta)$, there is a component of $L$ such that the entire boundary of the disk belongs to that component. All the bands emanating from a disk must belong to the same component of $L$ as the boundary of that disk, and so each connected component of~$F(\beta)$ must have only one component of $L$ (a knot) as its boundary. 

Let us first assume that we are in Case 2, so that $F(\beta)$ is a disjoint union of quasipositive Seifert surfaces each of which has a knot as its boundary. By assumption, $F(\beta)$ is not a union of disks. Let~$G$ be one of the connected components of $F(\beta)$ which is not a disk. We claim that we can choose~$Z$ and~$\varphi$ as in \eqref{eq:cylindertriple} such that $B_\beta = Z\cap F(\beta)$ is a band of $G \subset F(\beta)$ corresponding to a positive band word of $\beta$ and such that $\gamma = \varphi^{-1}(  C_{\frac{1}{3}}(\frac{1}{2})\times \{\frac{1}{2}\})$
is an unknot in $S^3 \setminus F(\beta)$ that does not bound a disk in $S^3 \setminus G$. \Cref{claim:mainlem} will then follow from the more general statement in \Cref{lem:step2}, which can be shown using a standard innermost circle argument. Note that quasipositive Seifert surfaces are of minimal genus \cite{rudolphslice, kronheimermrowka} and thus incompressible. For the reader's convenience, we will prove \Cref{lem:step2} below.

\begin{lem}\label{lem:step2}
Let $F$ be an incompressible Seifert surface for a link $L$ and let $\gamma \subset S^3 \setminus F$ be a simple closed curve. If there exists a disk in $S^3 \setminus L$ with boundary $\gamma$, then there also exists a disk in $S^3 \setminus F$ with boundary $\gamma$. 
\end{lem}

So if we find $Z$ and $\varphi$ as in \eqref{eq:cylindertriple} such that $\gamma = \varphi^{-1}(  C_{\frac{1}{3}}(\frac{1}{2})\times \{\frac{1}{2}\})$ does not bound a disk in $S^3 \setminus G$, it does not bound a disk in $S^3 \setminus \partial G$ either. To conclude the proof of \Cref{claim:mainlem} in Case 2, it remains to show that this is always possible. 

To that end, we claim that there exists a positive band word $\sigma_{i_\ell, j_\ell}$ in $\beta$ which fulfills the following condition: the core of the half-twisted band $B_\beta$ of $F(\beta)$ associated to $\sigma_{i_\ell, j_\ell}$ together with an arc in~$G$ the interior of which misses $B_\beta$ unite to a simple closed curve $\eta$ in $G$ so that $\eta$ and a meridian of $B_\beta$ have linking number $\pm 1$. Under a diffeomorphism $\varphi \colon Z \to D^2 \times S^1$ as in \eqref{eq:cylindertriple}, we can identify any of the half-twisted bands in $G$ with $\left[\frac{1}{3}, \frac{2}{3}\right] \times [0,1] \subset D^2 \times [0,1]$ for an appropriately chosen cylinder $Z \subset S^3$. A \emph{meridian} of a band $B_\beta$ of $G$ for us is then a simple closed curve in $S^3 \setminus G$ which is isotopic to $\varphi^{-1}(  C_{\frac{1}{3}}(\frac{1}{2})\times \{\frac{1}{2}\})$ under this identification and the \emph{core} of $B_\beta$ is $\varphi^{-1}\left(\{\frac{1}{2}\}\times \left[0,1\right]\right)$. If we find a band $B_\beta$ with the above requirements, the condition on the linking number of $\eta$ and the meridian $\gamma = \varphi^{-1}(  C_{\frac{1}{3}}(\frac{1}{2})\times\left\{\frac{1}{2}\right\})$ of $B_\beta$ will imply that $\gamma$ cannot bound a disk in $S^3 \setminus G$.

The quasipositive Seifert surface $G$ deformation retracts onto a graph $\Gamma$ in $S^3$ whose vertices correspond to the disks of $G$ and whose edges of $\Gamma$ correspond to the bands of $G$, respectively. Since $G$ is not a disk, $\Gamma$ is not a tree, hence there must exist an edge $e$ of $\Gamma$ such that $\Gamma \setminus e$ is not disconnected. This edge $e$ together with a path in $\Gamma$ connecting the vertices of $e$, but missing the interior of $e$, forms a simple closed curve in $\Gamma$ which has linking number $\pm 1$ with its meridian in~$S^3 \setminus \Gamma$. For the desired positive band word $\sigma_{i_\ell, j_\ell}$, we can take the one corresponding to the edge~$e$.

In summary, we have shown that in Case 2 we can choose $Z$ and $\varphi$ as in \eqref{eq:cylindertriple}, such that $B_\beta = Z\cap F(\beta)$ is a band of $F(\beta)$ corresponding to a positive band word of $\beta$ that is contained in one of these Seifert surfaces $G$ and such that its meridian $\gamma$ is an unknot in $S^3 \setminus F(\beta)$ that does not bound a disk in $S^3 \setminus G$ and therefore by \Cref{lem:step2} does not bound a disk in~$S^3 \setminus \partial G$ either.

Now suppose that we are in Case 1, so that we can choose a cylinder $Z \subset S^3$ and an orientation-preserving diffeomorphism $\varphi$ as in \eqref{eq:cylindertriple} such that $B_\beta = Z\cap F(\beta)$ is a band of $F(\beta)$ corresponding to a positive band word of $\beta$ where the two components of $\partial B_\beta \cap \partial F(\beta)$ belong to two different components of the link $\partial F(\beta)$. 
We claim that $\gamma = \varphi^{-1}(  C_{\frac{1}{3}}(\frac{1}{2})\times \{\frac{1}{2}\})$ is an unknot in $S^3 \setminus F(\beta)$ which does not bound a disk in $S^3 \setminus F(\beta)$ and thus, by \Cref{lem:step2}, does not bound a disk in $S^3 \setminus \partial F(\beta)$.

Similar as in the argument in Case 2 above, the quasipositive Seifert surface $F(\beta)$ deformation retracts onto a graph $\Gamma$ in $S^3$ whose vertices correspond to the disks of~$F(\beta)$ and whose edges correspond to the bands of~$F(\beta)$, respectively. Consider the edge $e$ of $\Gamma$ that corresponds to the band~$B_\beta$. Since the two components of $\partial B_\beta \cap \partial F(\beta)$ belong to two different components of~$\partial F(\beta)$, this edge must be part of a cycle in $\Gamma$. This cycle is a simple closed curve in $\Gamma$ which has linking number $\pm 1$ with its meridian in $S^3 \setminus \Gamma$, so the core of the band $B_\beta$ together with a certain arc in~$F(\beta)$ unite to form a simple closed curve in $F(\beta)$ which has linking number $\pm 1$ with the meridian of $B_\beta$ and the claim follows.
\end{claimproof}

\begin{proof}[Proof of \Cref{lem:step2}]
Let $D \subset S^3 \setminus L$ be a disk with $\partial D = \gamma \subset S^3 \setminus F$ and suppose that $D$ intersects~$F$ non-trivially. 
Up to an ambient isotopy, we can assume that $D$ and $F$ intersect transversally in~$S^3$~\cite{guillemin-pollack}. 
Then $D \cap F$ is a one-dimensional compact manifold, so a finite collection of simple closed curves. Using the $2$-dimensional Schoenflies theorem \cite[Section 2A]{rolfsen_2003}, each of these simple closed curves bounds a disk in $D$. Let $R$ be one of the simple closed curves in $D \cap F$ which is innermost in the sense that the interior of the disk $D^\prime$ bounded by $R$ in $D$ misses $F$. Since $F$ is incompressible, $R$ must also bound a disk $D^{\prime\prime}$ in $F$. The union of $D^\prime$ and $D^{\prime\prime}$ forms a $2$-sphere which, by the $3$-dimensional Schoenflies theorem \cite[Section 2F]{rolfsen_2003}, bounds a ball in $S^3$. We can push~$F$ along this ball to obtain a Seifert surface $F^\prime$ for $L$ which is ambient isotopic to $F$ and intersects~$D$ in less simple closed curves than $F$. We repeat this process until we obtain a Seifert surface $F^{\prime \prime}$ for~$L$ which is ambient isotopic to $F$ and disjoint from $D$. 
In summary, up to an ambient isotopy we found the desired disk in $S^3 \setminus F$.
\end{proof}
This finishes the proof of \Cref{lem:cylinder}.
\end{proof}

\section{Proof of \texorpdfstring{\Cref{thm:main}}{Theorem 1}}\label{sec:proofthm}

Recall the statement of \Cref{thm:main}: Every strongly quasipositive link other than an unlink is smoothly concordant to infinitely many pairwise non-isotopic strongly quasipositive links.

\begin{proof}[Proof of \Cref{thm:main}]
Let $L$ be a non-trivial link, let $F$ be a quasipositive Seifert surface for $L$ and let $C=\mnine$ such that the annulus $A\left(C, -1\right)$ is quasipositive (see \Cref{sec:qpannuli}). Let $\alpha$ be as in~\eqref{eq:alphabraidword} from \Cref{sec:qpannuli} such that $A\left(C, -1\right)$ is ambient isotopic to $F(\alpha)$, and, as in \Cref{sec:tyingknots}, choose a strongly quasipositive braid word $\beta \in B_n$, $n \geq 1,$ such that $F$ is ambient isotopic to~$F(\beta)$.

The statement of \Cref{thm:main} will follow from an iterative application of the operation defined in \Cref{sec:tyingknots}: Given two quasipositive Seifert surfaces $F(\alpha)$ and $F(\beta)$ for links $\partial F(\alpha)$ and $\partial F(\beta)$, respectively, using \Cref{lem:cylinder} and \Cref{lem:operation} we can construct a quasipositive Seifert surface $F^\prime$ with boundary a link $\partial F^\prime$ that is concordant, but not isotopic to $\partial F(\beta)$. We will denote this surface by $F(\alpha)\oplus F(\beta) \eqdef F^\prime$. We define $F_0 = F(\beta)$, $F_1 =F^\prime =  F(\alpha)\oplus F(\beta)$ and, inductively, $F_{i+1} = F(\alpha) \oplus F_{i}$ for all $i \geq 1$. The links $\{\partial F_i\}_{i \geq 0}$ are all in the same concordance class (the class of $L = \partial F_0 = \partial F(\beta)$), but pairwise non-isotopic. Let us make this more precise.
Recall that we construct $F_1 = F^\prime$ by tying the knot $C$ into a specific band $B_\beta$ of $F(\beta)$ (see \Cref{sec:tyingknots}) and obtained~$\partial F_1$ as a satellite link with pattern $\partial F(\beta)$ and companion $C$ (see the proof of \Cref{lem:operation}). The surfaces $F(\beta)$ and $F_1$ are both quasipositive Seifert surfaces (see \Cref{lem:operation}) that can again be put in a position where we can choose a cylinder $Z_1 \subset S^3$ and an orientation-preserving diffeomorphism $\varphi_1 \colon Z_1 \to D^2 \times [0,1]$ that satisfies a condition equivalent to the one in \eqref{eq:cylindertriple} from \Cref{sec:tyingknots} for $Z$ and~$\varphi$. Then, by \Cref{lem:cylinder} and \Cref{lem:operation}, we can choose $Z_1$ and $\varphi_1$ such that the surface $F_2$ obtained from $F_1$ by tying the knot $C$ into a specific band of $F_1$ is quasipositive and has as boundary $\partial F_2$ a link that is concordant, but not isotopic to $\partial F_1$. Inductively, $F_{i+1}$ is obtained from $F_i$ by tying~$C$ into a band of $F_i$ such that $\partial F_{i+1}$ is concordant, but not isotopic to $\partial F_i$. However, up to an ambient isotopy of $F_1$ and $F_2$, we can assume that for $F_2$ we tie $C$ into the same band of $F_1$ as $B_\beta$ of $F(\beta)$.\footnote{To make the term ``same band'' more precise, we could fix an abstract embedding of the surface $F(\beta)$ throughout.} Note that the additional assumptions about this band needed for an application of \Cref{lem:operation} (and proven in \Cref{lem:cylinder} for $F_0$) will still be satisfied for $F_1$. In particular, a meridian of that band (refer to the proof of \Cref{lem:cylinder} for a precise definition) will not bound a disk in $S^3 \setminus \partial F_1$. This comes from the fact that $F_0$ and $F_1$ deformation retract onto graphs $\Gamma_0$ and $\Gamma_1$ in $S^3$ (where vertices of~$\Gamma_i$ correspond to disks of the surfaces $F_j$ for $j\in \{0,1\}$, and edges to bands) which only differ in that one edge $e$ of $\Gamma_0$ is subdivided into a connected sequence of $v$ new vertices and $v+1$ edges $e_1, e_2, \dots, e_{v+1}$ (coming from the annulus $A(C,-1) \cong F(\alpha)$). The subgraph $\Gamma_0 \setminus e$ is connected if and only if $\Gamma_1 \setminus e_k$ is connected for any $k \in \{1, \dots, v+1\}$; and both the bands corresponding to $e_1$ and $e_{v+1}$ will do when applying \Cref{lem:operation}.

As in the proof of \Cref{claim:concordance} in the proof of \Cref{lem:operation}, we now distinguish two cases to show that the constructed links $\partial F_i$ are pairwise non-isotopic. If the two components of $\partial B_\beta \cap \partial F(\beta)$ belong to the same component $J$ of the link $\partial F(\beta)$, then we actually obtain $\partial F_2$ as $P_{J}(C \# C)$, the satellite link with pattern $P_J = J$, but companion $C \# C$. Inductively, we get $$\partial F_{i} = P_{J}(\underbrace{C \# C \dots \# C)}_{i \text{ times}}.$$ 
Since $C$ is not isotopic to the unknot, the connected sums of $i$ and $k$ copies of~$C$, respectively, are not isotopic for $i \neq k$ (\eg by arguing with the additivity of the nonzero Seifert genus of $C$). It thus follows from \cite[Theorem 1.1]{kounomotegi} that
\begin{align*}
\partial F_{i} = P_{J}(\underbrace{C \# C \dots \# C)}_{i \text{ times}} \qquad \text{and} \qquad \partial F_{k} = P_{J}(\underbrace{C \# C \dots \# C)}_{k \text{ times}}
\end{align*} 
are not isotopic for $i \neq k$. Again, it is important that the pattern $J$ has wrapping number strictly greater than $1$ under the assumptions of \Cref{lem:operation} for $F(\beta)$. If the two components of $\partial B_\beta \cap \partial F(\beta)$ belong to different components $L_1$ and $L_2$ of the link $\partial F(\beta)$, then, by induction, the link $\partial F_i$ has components 
\begin{align*}
L_1 \# \underbrace{C \# C \dots \# C}_{i \text{ times}} 
\qquad \text{and} \qquad L_2 \# \underbrace{C^r \# C^r \dots \# C^r}_{i \text{ times}}
\end{align*} 
and so again $\partial F_i$ and $\partial F_k$ are not isotopic for $i \neq k$.
\end{proof}

\begin{rem}\label{thm:mainsurfaces}
A careful generalization of our proof of \Cref{thm:main} and in particular \Cref{lem:operation} shows the following slightly stronger statement.
Let $F$ be a quasipositive Seifert surface for a link $L$ other than an unlink. Then there exists an infinite family $\left\{F_i \times [0,1]\right\}_{i \geq 1}$ of smoothly and properly embedded $3$-manifolds $F_i \times [0,1]$ in $S^3 \times [0,1]$ where each $F_i$ is a surface such that $F_i \times \{0\}= F \times \{0\}$, $ F_i \times \{1\}= F_i^\prime \times \{1\}$ for some quasipositive Seifert surface $F_i^\prime$ with boundary $\partial F_i^\prime $ that is non-isotopic to $ \partial F = L$, and such that the boundaries 
$\partial F_i^\prime $ and $ \partial F_j^\prime$ are non-isotopic for $i \neq j$.
\end{rem}

We conclude with the promised details on the construction in \Cref{rem:BakerHedden}.

\begin{rem}\label{rem:oldconstrlonger}
We elaborate on how to construct an infinite family of concordant, pairwise non-isotopic, strongly quasipositive knots using Whitehead doubles. Note once again that the statement of \Cref{thm:main} is stronger. Let~$C$ be a non-trivial slice knot with $\TB(C)=-1$, \eg $C=\mnine$ (see \Cref{sec:qpannuli}). For $m \geq 1$, let $K_m$ be the connected sum of $m$ copies of $C$. Then for every~$m \geq 1$, the knot $K_m$ is slice, since $C$ is, and by inductively using the formula $\TB(L_1 \# L_2) = \TB(L_1)+\TB(L_2)+1$ for any knots $L_1, L_2$ \cite{etnyre_honda,torisu}, we have $\TB(K_m) = -1$. Note that $K_m$ and $K_n$ are not isotopic for $m \neq n$, since $C$ is non-trivial. Using the notation from \cite{hedden-whitehead}, we now define $J_m \eqdef D_+\left(K_m,-1\right)$ as the positive $(-1)$-twisted Whitehead double of $K_m$. Then $\{J_m\}_{m \geq 1}$ is the desired infinite family. Indeed, using $\TB(K_m) \geq -1$, by work of Rudolph (see \eg~\cite[102.4]{rudolph_handbook}) each $J_m$ is strongly quasipositive. Moreover, as $K_m$ and $K_n$ are not isotopic for $m \neq n$, the knots $J_m$ and~$J_n$ are not isotopic either for such $m$ and $n$ \cite{kounomotegi}. On the other hand, $J_m$ and $J_n$ are concordant for every $m \neq n$ as $K_m$ and $K_n$ are. Indeed, as noted in the proof of \Cref{lem:operation}, the satellite operation induces a well-defined map on the concordance group of which taking the positive twisted Whitehead double is a special case.
\end{rem}

\bibliographystyle{alpha}
\bibliography{bibliography}

@book{rolfsen_2003,
	author = {Rolfsen, D.},
	doi = {10.1090/chel/346},
	publisher = {AMS Chelsea},
	title = {Knots and Links},
	year = {2003}
}

@article{Baker_2015,
	author = {Baker, K.~L.},
	doi = {10.1112/jtopol/jtv024},
	fjournal = {Journal of Topology},
	issn = {1753-8416},
	journal = {J. Topol.},
	mrclass = {57M25 (57N70)},
	mrnumber = {3465837},
	mrreviewer = {Laurence R. Taylor},
	number = {1},
	pages = {1–4},
	title = {A note on the concordance of fibered knots},
	url = {https://doi.org/10.1112/jtopol/jtv024},
	volume = {9},
	year = {2016}
}

@article{fox_1962,
	author = {Fox, R. H.},
	doi = {10.1016/0016-0032(62)90835-9},
	journal = {Topology of 3-manifolds and related topics (Proc. The Univ. of Georgia Institute, 1961)},
	number = {4},
	pages = {168–176 },
	title = {Some problems in knot theory},
	volume = {274},
	year = {1962}
}

@article{hedden05,
	author = {Hedden, M.},
	doi = {10.1142/S0218216510008017},
	fjournal = {Journal of Knot Theory and its Ramifications},
	issn = {0218-2165},
	journal = {J. Knot Theory Ramifications},
	mrclass = {57M27 (57M25 57N70)},
	mrnumber = {2646650},
	number = {5},
	pages = {617–629},
	title = {Notions of positivity and the {O}zsváth-{S}zabó concordance invariant},
	url = {https://doi.org/10.1142/S0218216510008017},
	volume = {19},
	year = {2010}
}

@inproceedings{litherland,
	author = {Litherland, R. A.},
	booktitle = {Topology of low-dimensional manifolds ({P}roc. {S}econd {S}ussex {C}onf., {C}helwood {G}ate, 1977)},
	mrclass = {57M25},
	mrnumber = {547456},
	mrreviewer = {Lee Rudolph},
	pages = {71–84},
	publisher = {Springer, Berlin},
	series = {Lecture Notes in Math.},
	title = {Signatures of iterated torus knots},
	volume = {722},
	year = {1979}
}

@article{alexander,
	author = {Alexander, J.~W.},
	doi = {10.1073/pnas.9.3.93},
	issn = {0027-8424 1091-6490},
	journal = {Proceedings of the National Academy of Sciences of the United States of America},
	month = mar,
	number = {3},
	pages = {93–95},
	risfield_0_db = {PubMed},
	risfield_1_la = {eng},
	risfield_2_an = {16576674},
	risfield_3_u1 = {16576674[pmid]},
	risfield_4_u2 = {PMC1085274[pmcid]},
	title = {A Lemma on Systems of Knotted Curves},
	url = {https://pubmed.ncbi.nlm.nih.gov/16576674 https://www.ncbi.nlm.nih.gov/pmc/articles/PMC1085274/},
	volume = {9},
	year = {1923}
}

@article{artin_1925,
	author = {Artin, E.},
	doi = {10.1007/bf02950718},
	journal = {Abhandlungen aus dem Mathematischen Seminar der Universität Hamburg},
	pages = {47–72},
	title = {Theorie der {Z}öpfe},
	volume = {4},
	year = {1925}
}

@article{boileau_orevkov_2001,
    AUTHOR = {Boileau, M. and Orevkov, S.},
     TITLE = {Quasi-positivit\'e{} d'une courbe analytique dans une boule
              pseudo-convexe},
   JOURNAL = {C. R. Acad. Sci. Paris S\'er. I Math.},
  FJOURNAL = {Comptes Rendus de l'Acad\'emie des Sciences. S\'erie I.
              Math\'ematique},
    VOLUME = {332},
      YEAR = {2001},
    NUMBER = {9},
     PAGES = {825--830},
      ISSN = {0764-4442},
   MRCLASS = {32Q65 (32T15 57M25)},
  MRNUMBER = {1836094},
MRREVIEWER = {Lee\ Rudolph},
       DOI = {10.1016/S0764-4442(01)01945-0},
       URL = {https://doi.org/10.1016/S0764-4442(01)01945-0},
}

@article{Baader_2017,
	author = {Baader, S. and Dehornoy, P. and Liechti, L.},
	doi = {10.1112/blms.12124},
	issn = {0024-6093},
	journal = {Bulletin of the London Mathematical Society},
	month = {Dec},
	number = {1},
	pages = {166–173},
	publisher = {Wiley},
	title = {Signature and concordance of positive knots},
	url = {http://dx.doi.org/10.1112/blms.12124},
	volume = {50},
	year = {2017}
}

@article{Rudolph_1983_alg_fcts,
    AUTHOR = {Rudolph, L.},
     TITLE = {Algebraic functions and closed braids},
   JOURNAL = {Topology},
  FJOURNAL = {Topology. An International Journal of Mathematics},
    VOLUME = {22},
      YEAR = {1983},
    NUMBER = {2},
     PAGES = {191--202},
      ISSN = {0040-9383},
   MRCLASS = {57M25 (14E15)},
  MRNUMBER = {683760},
MRREVIEWER = {J.\ S.\ Birman},
       DOI = {10.1016/0040-9383(83)90031-9},
       URL = {https://doi.org/10.1016/0040-9383(83)90031-9},
}

@book{milnor_book,
	author = {Milnor, J.},
	mrclass = {57.20 (14.00)},
	mrnumber = {0239612},
	mrreviewer = {J. P. Levine},
	pages = {iii+122},
	publisher = {Princeton University Press, Princeton, N.J.; University of Tokyo Press, Tokyo},
	series = {Ann. Math. Stud., No. 61},
	title = {Singular points of complex hypersurfaces},
	year = {1968},
	doi = {10.1017/S0013091500027097}
}

@misc{knotinfo,
	author = {Livingston, C. and Moore, A.~H.},
	howpublished = {\url{knotinfo.math.indiana.edu}},
	month = {June 12,},
	title = {KnotInfo: Table of Knot Invariants},
	year = {2024}
}

@incollection{birmanbrendle,
	author = {Birman, J.~S. and Brendle, T.~E.},
	booktitle = {Handbook of knot theory},
	doi = {10.1016/B978-044451452-3/50003-4},
	mrclass = {57M25 (20F36 55R80 57M27 57R17)},
	mrnumber = {2179260},
	mrreviewer = {Darren D. Long},
	pages = {19–103},
	publisher = {Elsevier B. V., Amsterdam},
	title = {Braids: a survey},
	url = {https://doi.org/10.1016/B978-044451452-3/50003-4},
	year = {2005}
}

@incollection{bennequin,
	author = {Bennequin, D.},
	booktitle = {Third {S}chnepfenried geometry conference, {V}ol. 1 ({S}chnepfenried, 1982)},
	mrclass = {58F18 (57R15)},
	mrnumber = {753131},
	mrreviewer = {Yakov Eliashberg},
	pages = {87–161},
	publisher = {Soc. Math. France, Paris},
	series = {Astérisque},
	title = {Entrelacements et équations de {P}faff},
	volume = {107},
	year = {1983}
}

@article{rudolphslice,
	author = {Rudolph, L.},
	doi = {10.1090/S0273-0979-1993-00397-5},
	fjournal = {American Mathematical Society. Bulletin. New Series},
	issn = {0273-0979},
	journal = {Bull. Amer. Math. Soc. (N.S.)},
	mrclass = {57M25 (32S55)},
	mrnumber = {1193540},
	mrreviewer = {Charles Livingston},
	number = {1},
	pages = {51–59},
	title = {Quasipositivity as an obstruction to sliceness},
	url = {https://doi.org/10.1090/S0273-0979-1993-00397-5},
	volume = {29},
	year = {1993}
}

@article{kronheimermrowka,
	author = {Kronheimer, P. B. and Mrowka, T. S.},
	doi = {10.1016/0040-9383(93)90051-V},
	fjournal = {Topology. An International Journal of Mathematics},
	issn = {0040-9383},
	journal = {Topology},
	mrclass = {57R57 (57N13 57R40 57R55 58D29)},
	mrnumber = {1241873},
	mrreviewer = {Ronald J. Stern},
	number = {4},
	pages = {773–826},
	title = {Gauge theory for embedded surfaces. {I}},
	url = {https://doi.org/10.1016/0040-9383(93)90051-V},
	volume = {32},
	year = {1993}
}

@article{ng_thurston-bennequin,
	author = {Ng, L.~L.},
	doi = {10.2140/agt.2001.1.427},
	fjournal = {Algebraic \& Geometric Topology},
	issn = {1472-2747},
	journal = {Algebr. Geom. Topol.},
	mrclass = {57M27 (57M15 57M25)},
	mrnumber = {1852765},
	mrreviewer = {Peiyi Zhao},
	pages = {427–434},
	title = {Maximal {T}hurston-{B}ennequin number of two-bridge links},
	url = {https://doi.org/10.2140/agt.2001.1.427},
	volume = {1},
	year = {2001}
}

@incollection{etnyre,
	author = {Etnyre, J.~B.},
	booktitle = {Handbook of knot theory},
	doi = {10.1016/B978-044451452-3/50004-6},
	mrclass = {57R17 (53D35 57M25 57M27)},
	mrnumber = {2179261},
	mrreviewer = {Lenhard L. Ng},
	pages = {105–185},
	publisher = {Elsevier B. V., Amsterdam},
	title = {Legendrian and transversal knots},
	url = {https://doi.org/10.1016/B978-044451452-3/50004-6},
	year = {2005}
}

@article{hedden-whitehead,
	author = {Hedden, M.},
	doi = {10.2140/gt.2007.11.2277},
	fjournal = {Geometry \& Topology},
	issn = {1465-3060},
	journal = {Geom. Topol.},
	mrclass = {57M27 (57R58)},
	mrnumber = {2372849},
	mrreviewer = {Thomas E. Mark},
	pages = {2277–2338},
	title = {Knot {F}loer homology of {W}hitehead doubles},
	url = {https://doi.org/10.2140/gt.2007.11.2277},
	volume = {11},
	year = {2007}
}

@article{kounomotegi,
	author = {Kouno, M. and Motegi, K.},
	doi = {10.1017/S0305004100072054},
	fjournal = {Mathematical Proceedings of the Cambridge Philosophical Society},
	issn = {0305-0041},
	journal = {Math. Proc. Cambridge Philos. Soc.},
	mrclass = {57M25},
	mrnumber = {1277057},
	number = {2},
	pages = {219–228},
	title = {On satellite knots},
	url = {https://doi.org/10.1017/S0305004100072054},
	volume = {115},
	year = {1994}
}

@article{borodzikFeller,
	author = {Borodzik, M. and Feller, P.},
	doi = {10.1016/j.matpur.2019.03.001},
	fjournal = {Journal de Mathématiques Pures et Appliquées. Neuvième Série},
	issn = {0021-7824},
	journal = {J. Math. Pures Appl. (9)},
	mrclass = {57K10},
	mrnumber = {4030255},
	mrreviewer = {Sebastian Baader},
	pages = {273–279},
	title = {Up to topological concordance, links are strongly quasipositive},
	url = {https://doi.org/10.1016/j.matpur.2019.03.001},
	volume = {132},
	year = {2019}
}

@incollection{rudolph_handbook,
	author = {Rudolph, L.},
	booktitle = {Handbook of knot theory},
	doi = {10.1016/B978-044451452-3/50009-5},
	mrclass = {57M25 (32S55 57R17)},
	mrnumber = {2179266},
	mrreviewer = {J. S. Birman},
	pages = {349–427},
	publisher = {Elsevier B. V., Amsterdam},
	title = {Knot theory of complex plane curves},
	url = {https://doi.org/10.1016/B978-044451452-3/50009-5},
	year = {2005}
}

@article{etnyre_honda,
	author = {Etnyre, J.~B. and Honda, K.},
	doi = {10.1016/S0001-8708(02)00027-0},
	fjournal = {Advances in Mathematics},
	issn = {0001-8708},
	journal = {Adv. Math.},
	mrclass = {57M25 (57M50 57R17)},
	mrnumber = {2004728},
	mrreviewer = {Quach thi Câm Vân},
	number = {1},
	pages = {59–74},
	title = {On connected sums and {L}egendrian knots},
	url = {https://doi.org/10.1016/S0001-8708(02)00027-0},
	volume = {179},
	year = {2003}
}

@article{torisu,
	author = {Torisu, I.},
	doi = {10.2140/pjm.2003.210.359},
	fjournal = {Pacific Journal of Mathematics},
	issn = {0030-8730},
	journal = {Pacific J. Math.},
	mrclass = {57M27 (57M25)},
	mrnumber = {1988540},
	mrreviewer = {Mattia Mecchia},
	number = {2},
	pages = {359–365},
	title = {On the additivity of the {T}hurston-{B}ennequin invariant of {L}egendrian knots},
	url = {https://doi.org/10.2140/pjm.2003.210.359},
	volume = {210},
	year = {2003}
}

@article{rudolph_obstrtoslice,
	author = {Rudolph, L.},
	doi = {10.1007/BF01245177},
	fjournal = {Inventiones Mathematicae},
	issn = {0020-9910},
	journal = {Invent. Math.},
	mrclass = {57M25 (53C15 57R57)},
	mrnumber = {1309974},
	mrreviewer = {Ricardo F. Vila Freyer},
	number = {1},
	pages = {155–163},
	title = {An obstruction to sliceness via contact geometry and “classical” gauge theory},
	url = {https://doi.org/10.1007/BF01245177},
	volume = {119},
	year = {1995}
}

@article{Chantraine15,
	author = {Chantraine, B.},
	doi = {10.4171/QT/68},
	fjournal = {Quantum Topology},
	issn = {1663-487X},
	journal = {Quantum Topol.},
	mrclass = {57R17 (53D42 57M27 57M50)},
	mrnumber = {3392961},
	mrreviewer = {Daniel V. Mathews},
	number = {3},
	pages = {451–474},
	title = {Lagrangian concordance is not a symmetric relation},
	url = {https://doi.org/10.4171/QT/68},
	volume = {6},
	year = {2015}
}

@article{rudolph_quasipositive_annuli,
	author = {Rudolph, L.},
	doi = {10.1142/S0218216592000227},
	fjournal = {Journal of Knot Theory and its Ramifications},
	issn = {0218-2165},
	journal = {J. Knot Theory Ramifications},
	mrclass = {57M25},
	mrnumber = {1194997},
	mrreviewer = {J. S. Birman},
	number = {4},
	pages = {451–466},
	title = {Quasipositive annuli. ({C}onstructions of quasipositive knots and links. {IV})},
	url = {https://doi.org/10.1142/S0218216592000227},
	volume = {1},
	year = {1992}
}

@article{nakamura,
	author = {Nakamura, T.},
	fjournal = {Osaka Journal of Mathematics},
	issn = {0030-6126},
	journal = {Osaka J. Math.},
	mrclass = {57M25},
	mrnumber = {1772843},
	mrreviewer = {Kimihiko Motegi},
	number = {2},
	pages = {441–451},
	title = {Four-genus and unknotting number of positive knots and links},
	url = {http://projecteuclid.org/euclid.ojm/1200789208},
	volume = {37},
	year = {2000}
}

@inproceedings{Rudolph_positiveLinksSQP,
	author = {Rudolph, L.},
	booktitle = {Proceedings of the {K}irbyfest ({B}erkeley, {CA}, 1998)},
	doi = {10.2140/gtm.1999.2.555},
	mrclass = {57M25},
	mrnumber = {1734423},
	pages = {555–562},
	publisher = {Geom. Topol. Publ., Coventry},
	series = {Geom. Topol. Monogr.},
	title = {Positive links are strongly quasipositive},
	url = {https://doi.org/10.2140/gtm.1999.2.555},
	volume = {2},
	year = {1999}
}

@book{guillemin-pollack,
	author = {Guillemin, V. and Pollack, A.},
	doi = {10.1090/chel/370},
	isbn = {978-0-8218-5193-7},
	mrclass = {58-01 (57-01)},
	mrnumber = {2680546},
	note = {Reprint of the 1974 original},
	pages = {xviii+224},
	publisher = {AMS Chelsea Publishing, Providence, RI},
	title = {Differential topology},
	url = {https://doi.org/10.1090/chel/370},
	year = {2010}
}

@article{StoimenowConc,
	author = {Stoimenow, A.},
	doi = {10.1142/S0129167X15500500},
	fjournal = {International Journal of Mathematics},
	issn = {0129-167X},
	journal = {Internat. J. Math.},
	mrclass = {57M25 (11C20 15A63 57M27 57N70)},
	mrnumber = {3357039},
	mrreviewer = {Christopher William Davis},
	number = {7},
	pages = {1550050, 36},
	title = {Application of braiding sequences {III}: {C}oncordance of positive knots},
	url = {https://doi.org/10.1142/S0129167X15500500},
	volume = {26},
	year = {2015}
}

@article{Stoimenowpossign,
	author = {Stoimenow, A.},
	doi = {10.1090/S0002-9947-08-04410-3},
	fjournal = {Transactions of the American Mathematical Society},
	issn = {0002-9947},
	journal = {Trans. Amer. Math. Soc.},
	mrclass = {57M25 (57M27 57N70)},
	mrnumber = {2415070},
	mrreviewer = {Jae Choon Cha},
	number = {10},
	pages = {5173–5199},
	title = {Bennequin's inequality and the positivity of the signature},
	url = {https://doi.org/10.1090/S0002-9947-08-04410-3},
	volume = {360},
	year = {2008}
}

@article{Rudolph_surfaces,
	author = {Rudolph, L.},
	doi = {10.1016/0040-9383(92)90017-C},
	fjournal = {Topology. An International Journal of Mathematics},
	issn = {0040-9383},
	journal = {Topology},
	mrclass = {57M25 (20F36)},
	mrnumber = {1167166},
	mrreviewer = {J. S. Birman},
	number = {2},
	pages = {231–237},
	title = {Constructions of quasipositive knots and links. {III}. {A} characterization of quasipositive {S}eifert surfaces},
	url = {https://doi.org/10.1016/0040-9383(92)90017-C},
	volume = {31},
	year = {1992}
}

@incollection{Rudolph_II,
	author = {Rudolph, L.},
	booktitle = {Four-manifold theory ({D}urham, {N}.{H}., 1982)},
	doi = {10.1090/conm/035/780596},
	mrclass = {57M25},
	mrnumber = {780596},
	mrreviewer = {J. S. Birman},
	pages = {485–491},
	publisher = {Amer. Math. Soc., Providence, RI},
	series = {Contemp. Math.},
	title = {Constructions of quasipositive knots and links. {II}},
	url = {https://doi.org/10.1090/conm/035/780596},
	volume = {35},
	year = {1984}
}

@book{brieskornknoerrer,
	author = {Brieskorn, E. and Knörrer, H.},
	doi = {10.1007/978-3-0348-5097-1},
	isbn = {3-7643-1769-8},
	mrclass = {14-01 (14H20 14Hxx)},
	mrnumber = {886476},
	note = {Translated from the German by John Stillwell},
	pages = {vi+721},
	publisher = {Birkhäuser Verlag, Basel},
	title = {Plane algebraic curves},
	url = {https://doi.org/10.1007/978-3-0348-5097-1},
	year = {1986}
}

@article{rudolph_braidedsurfaces,
	author = {Rudolph, L.},
	doi = {10.1007/BF02564622},
	fjournal = {Commentarii Mathematici Helvetici},
	issn = {0010-2571},
	journal = {Comment. Math. Helv.},
	mrclass = {57M25 (32E10)},
	mrnumber = {699004},
	mrreviewer = {G. Peter Scott},
	number = {1},
	pages = {1–37},
	title = {Braided surfaces and {S}eifert ribbons for closed braids},
	url = {https://doi.org/10.1007/BF02564622},
	volume = {58},
	year = {1983}
}

@article{kuessner,
	author = {Kuessner, T},
	mrclass = {Thesis},
	note = {Thesis (Universität Tüubingen)},
	title = {Relative simplicial volume},
	year = {2001}
}

@article{gromovvolume,
     author = {Gromov, M.},
     title = {Volume and bounded cohomology},
     journal = {Publications Mathématiques de l'IHÉS},
     pages = {5-99},
     publisher = {Institut des Hautes \'Etudes Scientifiques},
     volume = {56},
     year = {1982},
     mrnumber = {84h:53053},
     zbl = {0516.53046},
     language = {en}
}

@article{Soma1981,
author = {Soma, T.},
journal = {Inventiones mathematicae},
pages = {445-454},
title = {The {G}romov {I}nvariant of {L}inks.},
url = {http://eudml.org/doc/142831},
volume = {64},
year = {1981},
}

@book{thurston2022geometry,
  title={The Geometry and Topology of Three-Manifolds: With a Preface by Steven P. Kerckhoff},
  author={Thurston, W.P.},
  isbn={9781470463915},
  lccn={2021037702},
  series={Collected Works},
  url={https://books.google.de/books?id=dKF9EAAAQBAJ},
  year={2022},
  publisher={American Mathematical Society}
}

@article {jacoshalen,
    AUTHOR = {Jaco, W.~H. and Shalen, P.~B.},
     TITLE = {Seifert fibered spaces in {$3$}-manifolds},
   JOURNAL = {Mem. Amer. Math. Soc.},
  FJOURNAL = {Memoirs of the American Mathematical Society},
    VOLUME = {21},
      YEAR = {1979},
    NUMBER = {220},
     PAGES = {viii+192},
      ISSN = {0065-9266,1947-6221},
   MRCLASS = {57N10},
  MRNUMBER = {539411},
MRREVIEWER = {Hugh\ M.\ Hilden},
       DOI = {10.1090/memo/0220},
       URL = {https://doi.org/10.1090/memo/0220},
}

@book {johannson,
    AUTHOR = {Johannson, K.},
     TITLE = {Homotopy equivalences of {$3$}-manifolds with boundaries},
    SERIES = {Lecture Notes in Mathematics},
    VOLUME = {761},
 PUBLISHER = {Springer, Berlin},
      YEAR = {1979},
     PAGES = {ii+303},
      ISBN = {3-540-09714-7},
   MRCLASS = {57N10},
  MRNUMBER = {551744},
MRREVIEWER = {John\ Hempel},
}

@article {birmankolee,
    AUTHOR = {Birman, J. and Ko, K.~H. and Lee, S.~J.},
     TITLE = {A new approach to the word and conjugacy problems in the braid
              groups},
   JOURNAL = {Adv. Math.},
  FJOURNAL = {Advances in Mathematics},
    VOLUME = {139},
      YEAR = {1998},
    NUMBER = {2},
     PAGES = {322--353},
      ISSN = {0001-8708,1090-2082},
   MRCLASS = {20F36 (20F05 20F10)},
  MRNUMBER = {1654165},
MRREVIEWER = {Yuji\ Kobayashi},
       DOI = {10.1006/aima.1998.1761},
       URL = {https://doi.org/10.1006/aima.1998.1761},
}

@article{bucheretal,
author = {Bucher, M. and Burger, M. and Frigerio, R. and Iozzi, A. and Pagliantini, C. and Pozzetti, M. B.},
title = {Isometric embeddings in bounded cohomology},
journal = {Journal of Topology and Analysis},
volume = {06},
number = {01},
pages = {1-25},
year = {2014},
doi = {10.1142/S1793525314500058},

URL = { 
    
        https://doi.org/10.1142/S1793525314500058
    
    

},
eprint = { 
    
        https://doi.org/10.1142/S1793525314500058
    
    

}
}

\end{document}